\documentclass[10pt, reqno]{amsart}

\usepackage{amsmath}
\usepackage{amsthm}
\usepackage{hyperref}
\usepackage{ragged2e}
\usepackage{enumitem}
\usepackage{tikz-cd}
\usepackage{mathtools}
\usepackage{accents}

\usepackage{amssymb}
\newcommand{\bb}{\mathbb}

\newcommand{\vphi}{\varphi}

\newcommand{\Res}{\operatorname{Res}}

\newcommand{\ovl}{\overline}

\newcommand{\h}{\widehat{h}}
	
\newtheorem{theorem}{Theorem}
\numberwithin{theorem}{section}
\newtheorem{lemma}[theorem]{Lemma}
\newtheorem{proposition}[theorem]{Proposition}
\numberwithin{equation}{section}

\numberwithin{example}{section}

\numberwithin{definition}{section}

\makeindex
\title{$S$-Integral Points in Orbits on $\bb{P}^1$}
\begin{document}

\author{Jit Wu Yap}

\begin{abstract}
Let $K$ be a number field and $S$ a finite set of places of $K$ that contains  all of the archimedean places. Let $\vphi: \bb{P}^1 \to \bb{P}^1$ be a rational map of degree $d \geq 2$ defined over $K$. Given $\alpha \in \bb{P}^1(K)$ non-preperiodic and $\beta \in \bb{P}^1(K)$ non-exceptional, we prove an upper bound of the form $O(|S|^{1+\epsilon})$ on the number of points in the forward orbit of $\alpha$ that are $S$-integral relative to $\beta$, extending results of Hsia--Silverman \cite{HS11}. We also prove uniform bounds when $\vphi$ is a polynomial, extending results of Krieger et al \cite{KLS15}.     
\end{abstract}

\maketitle

\tableofcontents

\section{Introduction}
Let $K$ be a number field and $S$ a finite set of places containing all archimedean places of $K$. Let $\vphi: \bb{P}^1 \to \bb{P}^1$ be a rational map defined over $K$ of degree $d \geq 2$ and let $\beta \in \bb{P}^1(K)$ be a non-exceptional point, i.e. $|\vphi^{-n}(\beta)| \to \infty$ as $n \to \infty$. For a point $\alpha \in \bb{P}^1(K)$, Silverman \cite{Sil93} has shown that only finitely many points in its forward orbit $\{\vphi^n(\alpha)\}$ can be $S$-integral relative to $\beta$. Here, we say $x,y \in \bb{P}^1(K)$ are $S$-integral relative to each other if for all $v \not \in S$, the reductions $\ovl{x},\ovl{y} \in \bb{P}^1(k_v)$ are distinct, where $k_v$ denotes the residue field of $O_K$ at the place $v$.
\par 
In this paper, we study the number of $S$-integral points that can exist in the forward orbit of $\alpha$. In \cite{HS11}, Hsia and Silverman proved a bound of $c \, 4^{|S|}$ for some constant $c > 0$ using Roth's theorem. More recently in \cite{KLS15}, Krieger et al also proved an exponential bound $c^{|S|}$ on the number of $S$-integral points using Siegel's theorem. Our first theorem is an improved upper bound on the number of $S$-integral points relative to $\beta$ among $\{\vphi^n(\alpha)\}_{n \geq 0}$. 

\begin{theorem} \label{IntroTheorem1}
Let $K$ be a number field and $\vphi$ a rational map defined over $K$ with degree $d \geq 2$. Then there exists a constant $c_1 = c_1(K,\vphi) > 0$ such that for any non-preperiodic $\alpha \in \bb{P}^1(K)$, non-exceptional $\beta \in \bb{P}^1(K)$ and any finite set of places $S$ of $K$ containing all archimedean places, we have
$$\left| \{ n \geq 0 \mid \vphi^n(\alpha) \text{ is S-integral relative to } \beta \} \right| \leq c_1 |S| (\log |S|+1)^5.$$
\end{theorem}

This gives us a superlinear bound on the number of $S$-integral points.  Our approach follows that of \cite{Sil93} and \cite{HS11}. In \cite{HS11}, Hsia and Silverman showed the existence of a constant $\gamma = \gamma(d,[K:\bb{Q}])$ for which the number of points $S$-integral relative to $\beta$ in $\{\vphi^n(\alpha)\}$ is less than
$$(\gamma) \cdot 4^{|S|}  + \log_d^+ \left( \frac{h(\vphi) + \h_{\vphi}(\beta)}{\h_{\vphi}(\alpha)} \right)$$
where $\h_{\vphi}$ is the canonical height for the rational map $\vphi$ (see Section \ref{sec:Definition}). 
Similarly we also obtain an upper bound of this form, upon which Theorem \ref{IntroTheorem1} follows. 

\begin{theorem} \label{IntroTheorem3}
Keep the notations of Theorem \ref{IntroTheorem1}. Then there exists a constant $c_2 = c_2(d)$ such that the following is true: 
$$|\{ n \geq 0 \mid \vphi^n(\alpha) \text{ is S-integral relative to } \beta\}| \leq c_2 |S| (\log |S|+1)^6 + \log_d^+ \left(\frac{h(\vphi) + 1}{\h_{\vphi}(\alpha)} \right).$$
\end{theorem}

Theorem \ref{IntroTheorem1} follows easily from Theorem \ref{IntroTheorem3} since if $K$ and $\vphi$ are fixed, then 
$$\inf_{\alpha \in K, \h_{\vphi}(\alpha) \not = 0} \h_{\vphi}(\alpha) > 0$$
due to the Northcott property for heights. Compared to Hsia--Silverman's bound, other than the improvement as a function of $|S|$, we have also removed the dependency on $\beta$. 
\par 
It has already been observed in \cite{KLS15} that it is impossible to obtain  bounds for $S$-integral points that are independent of $\vphi$. We show in Section \ref{sec: Uniform} that for $S$-integral points, the dependency of Theorem \ref{IntroTheorem3} in $\frac{h(\vphi)}{\h_{\vphi}(\alpha)}$ cannot be removed. In many cases, one knows the existence of a uniform constant $c > 0$ such that $\h_{\vphi}(\alpha) > c h(\vphi)$ which would lead to uniformity results. For example, \cite{KLS15} asks if a uniform bound for points in an orbit that are $S$-units can be achieved, and proves a uniform bound for monic polynomials $\vphi(z) \in O_S[z]$. Using results of Looper \cite{Loo19}, we generalize \cite[Theorem 1.7]{KLS15} to polynomials with at most $s$ places of bad reduction.

\begin{theorem} \label{IntroTheorem4}
Keep the notation of Theorem \ref{IntroTheorem1}. Assume that $\vphi(z) \not = cz^{d}$. Then there exists a constant $c_3 = c_3([K:\bb{Q}],d,s,|S|) > 0$ such that if $\vphi(z) \in K[z]$ is a polynomial of degree $d$ that has $s$ places of bad reduction, then 
$$\left| \{n \geq 0 \mid \vphi^n(\alpha) \text{ is an S-unit } \}\right| \leq c_3.$$
\end{theorem}

% We also study the special case where $\alpha = \beta$ is a non-preperiodic critical point of $\vphi$. This situation has been studied before in the context of Zsigmondy sets by Krieger \cite{Kri13} and Ren \cite{Ren21}. Using the dynamical properties of a critical point, we are able to obtain a linear $|S| + o(|S|)$ bound.

% \begin{theorem} \label{IntroTheorem5}
% There exists a constant $c_7 > 0$ such that the following holds: let $K$ be a number field, $S$ a finite set of places containing the archimedean places and $\vphi : \bb{P}^1 \to \bb{P}^1$ a rational map defined over $K$ of degree $d \geq 2$. Let $\alpha \in \bb{P}^1(K)$ be a critical point of $\vphi$ that is not preperiodic. Then
% $$\left| \{n \geq 0 \mid \vphi^n(\alpha) \text{ is S-integral relative to } \alpha \}\right| \leq |S| +  c_7 \log (|S|+1) + \log^+_d \left(\frac{h(\vphi)+1}{\h_{\vphi}(\alpha)} \right).$$ 
% \end{theorem}

\subsection{Related Results} There have been many results regarding finiteness of $S$-integral points and quantitative bounds on them, beginning with Siegel's theorem on integral points \cite{Sie29}, which as a corollary proves that the unit equation
$$u+v = 1$$
has only finitely many solutions for $S$-units $u,v$ in a fixed number field $O_K$. There has been much study of how many solutions the $S$-unit equation and its variants has and whether they can be effectively determined, for example see \cite{Eve84, ESS02, EG15}.
\par 
For Siegel's theorem, a quantitative bound was given by Silverman \cite{Sil87b}. This lead quantitative bounds on the number of $S$-integral points an elliptic curve $E$ can have by Gross--Silverman \cite{GS95}, building on work of Hindry--Silverman \cite{HS88}.  
\par 
In \cite{Sil93}, Silverman proves a dynamical analogue of Siegel's theorem. Given an endomorphism $\vphi: \bb{P}^1 \to \bb{P}^1$ of degree $d \geq 2$ over a number field $K$ and a non-exceptional $\beta$, for a non-preperiodic $\alpha$ there are only finitely many $n$'s for which $\vphi^{n}(\alpha)$ is $S$-integral relative to $\beta$. It still remains a widely open problem to generalize Silverman's result to higher dimensions, with many results conditional on some form of Vojta's or Lang--Vojta's conjecture \cite{Yas14, CSTZ15, Yas15, LY19, Mat23, GN24, Mat24}. 
\par 
Similar in spirit to \cite{GS95}, Hsia and Silverman \cite{HS11} gave a quantitative estimate on the number of $n$'s for which $\vphi^n(\alpha)$ is $S$-integral relative to $\beta$. We note that Mello \cite{Mel19} has generalized this to the semigroup setting. For both results, their bound is exponential in $|S|$. Ingram \cite{Ing09} also has some effective results for orbits on elliptic curves.
\par 
For backward orbits, there are also similar results regarding finiteness and quantitative bounds of $S$-integral points, which was first studied by Baker--Ih--Rumely \cite{BIR08} in the case of $\bb{G}_m$ and elliptic curves. For example, if one consider all points in the preimage $\vphi^{-n}(\beta)$, then finiteness of $S$-integrality is proven by Szpiro--Tucker \cite{ST11} and also by Sookdeo \cite{Soo11}. Quantitative or uniform results have been established by Young \cite{You22}, Padhy--Rout \cite{PR24} and the first author \cite{Yap25}. For backward orbits, quantitative results are obtained by using Favre--Rivera-Letelier's quantitative equidistribution result \cite{FRL06}. 

\subsection{Proof sketch}
We now give a sketch of our arguments. Let   
$$\lambda_v([x_0:x_1],[y_0:y_1]) = -\log\frac{|x_0y_1 - x_1y_0|_v}{\max\{|x_0|_v,|x_1|_v\} \max\{|y_0|_v,|y_1|_v\}}$$
be the logarithmic chordal distance. Then if $\alpha$ is $S$-integral relative to $\beta$, we have $\lambda_v(\alpha,\beta) = 0$
for all $v \not \in S$. Fix some positive $\delta > 0$. To obtain finiteness, Hsia and Silverman show that if $\vphi^n(\alpha)$ is $S$-integral relative to $\beta$ and $n$ is sufficiently large, then
$$\sum_{v \in S} N_v \lambda_v(\vphi^{n-m}(\alpha),\beta_{v,m}) \geq (2+\delta) h(\vphi^{n-m}(\alpha))$$
where $m$ is some fixed natural number depending only on the degree $d$ and $\beta_{v,m}$ are elements of $\vphi^{-m}(\beta)$. Applying a quantitative version of Roth's theorem \cite[Theorem 10]{HS11} that allows summing over a finite set of places $S$, they obtain an upper bound on the number of $S$-integral points of the form $(\gamma) \cdot 4^{|S|}$.
\par 
In our proof, for each $n$ where $\vphi^n(\alpha)$ is $S$-integral relative to $\beta$, we instead show there exists a place $v_n \in S$ such that the inequality
$$N_{v_n} \log^+|\vphi^{n-m'}(\alpha)-\beta_{v_n,m'}|^{-1}_{v_n} \geq (2+\delta)h(\vphi^{n-m'}(\alpha)),$$
where $m'$ now depends not only on $d$ but also on $|S|$. The advantage is that we can now apply a quantitative Roth's theorem (Theorem \ref{QuantitativeRoth}) that only deals with a single place $v$, but in return provides us with a much better upper bound on the number of exceptions, as we will now explain.  
\par 
For simplicity, let's first assume that $\beta$ is not in the forward orbit of a critical point; i.e. $\vphi^n(x) = \beta$ has $d^n$ distinct solutions for all $n$. We let $N$ be the smallest positive integer such that $d^N > (2+\delta)|S|$. Then since 
$$\sum_{v \in M_K} N_v \log^+|\vphi^n(\alpha) - \beta|^{-1}_v = h(\vphi^n(\alpha) - \beta) \approx h(\vphi^n(\alpha)),$$
if $\vphi^n(\alpha)$ is $S$-integral relative to $\beta$, we get
$$\sum_{v \in S} N_v \log^+|\vphi^n(\alpha) -\beta|^{-1}_v \approx h(\vphi^n(\alpha)).$$
Hence there exists a place $v$ such that
\begin{equation} \label{eq: IntroInequality1}
N_v \log^+|\vphi^n(\alpha) - \beta|_v^{-1} \geq \frac{1}{|S|} h(\vphi^n(\alpha)).
\end{equation}
This implies that $\vphi^{n-N}(\alpha)$ is equally close, with some negligible error when $n$ is sufficiently large, to a root of $\vphi^N(x) = \beta$, say $\beta_n$. Then
$$N_v \log^+|\vphi^{n-N}(\alpha) - \beta_n|_v^{-1} \geq \frac{1}{|S|} h(\vphi^n(\alpha)) \geq \frac{d^N}{|S|}  h(\vphi^{n-N}(\alpha)) \geq (2+\delta) h(\vphi^{n-N}(\alpha)),$$
allowing us to apply Roth's theorem for a single place $v$. For a fixed $\beta' \in \vphi^{-N}(\beta)$ and $v \in S$, the number of exceptions to the inequality of Roth's theorem is at most $c \log ( [K(\beta'):K])^3 \leq c (d \log |S|)^3$, and each $n$ where $\vphi^n(\alpha)$ is $S$-integral relative to $\beta$ gives us such an exception. Then since there are at most $(2+\delta)d|S|$ elements in $\vphi^{-N}(\beta)$ and $|S|$ places, a pigeonhole argument implies that we can have at most $c |S|^2 (d \log |S|)^3$ values of $n$ for which $\vphi^n(\alpha)$ is $S$-integral relative to $\beta$, giving us a super-quadratic bound in $|S|$. Bringing it down to a superlinear bound requires some optimization and analyzing for how many distinct places $v$ can we obtain (\ref{eq: IntroInequality1}) for a given $n$.  
\par 
If $\beta$ is on the forward orbit of a critical point but not part of a superattracting cycle, the same argument works. For the case of a superattracting cycle, we instead analyze the growth of $N_v \log^+|\vphi^n(\alpha)-\beta|_v^{-1}$ directly, using the fact that $\beta$ attracts orbits at a fixed rate. This is done in Section \ref{sec: Superattracting} and the estimates are more technical. However, we do not invoke any deep results from diophantine approximation such as Roth's theorem and so this case can be considered more elementary.

\section{Background and Notations} \label{sec:Definition}
Let $K$ be a number field and $M_K$ the set of places of $K$. We define the Weil height $h$ for $x \in K$ by 
$$h(x) = \sum_{v \in M_K} N_v \log^+|x|_v$$
where $N_v = \frac{[K_v:\bb{Q}_p]}{[K:\bb{Q}]}$ with $p$ being the prime that $v$ lies above. For $x \in \ovl{\bb{Q}}$, this is independent of the choice of $K$ for which $x$ lies in so we may omit the field $K$. This definition extends in general to points in $\bb{P}^n(K)$ by
$$h([x_0:\cdots:x_n]) = \sum_{v \in M_K} N_v \log \max\{|x_0|_v,\ldots,|x_n|_v\}.$$
The product formula ensures that this definition is independent of the choice of homogeneous coordinates.
\par 
Let $\vphi: \bb{P}^1 \to \bb{P}^1$ be a rational map of degree $d \geq 2$ defined over $K$. Then the canonical height $\h_{\vphi}$ is given by
$$\h_{\vphi}(x) = \lim_{n \to \infty} \frac{h(\vphi^n(x))}{d^n}.$$
It is the unique function $\h_{\vphi}: \bb{P}^1(\ovl{K}) \to \bb{R}$ such that $\h_{\vphi}(\vphi(x)) = d \h_{\vphi}(x)$ and $|\h_{\vphi}(x) - h(x)|$ is uniformly bounded over all $x \in \bb{P}^1(\ovl{K})$. 
\par 
For a rational map $\vphi$ defined over $K$, we may write it as $f(z)/g(z)$ with $f(z) = \sum_{i=0}^{d} a_i z^i$ and $g(z) = \sum_{i=0}^{d} b_i z^i$ for some choice of $f,g \in K[z]$. We then define the height of $\vphi$ as
$$h(\vphi) = h([a_0:\cdots:a_d:b_0:\cdots:b_d]).$$
This is independent of the choice of $f$ and $g$ as our presentation for $\vphi$. If we normalize one of the coefficients among $a_i$ and $b_j$ to be $1$, then it follows from the definitions that 
$$h(\vphi) \geq h(a_i),h(b_j).$$
For any place $v \in M_K$, we define the chordal distance \cite[Chapter 5.1]{Ben19}
$$\rho_v([x_0:x_1],[y_0:y_1]) = \frac{|x_0y_1 - x_1y_0|_v}{\max\{|x_0|_v,|x_1|_v\} \max\{|y_0|_v,|y_1|_v\}}.$$
This is defined on $\bb{P}^1(\bb{C}_v) \times \bb{P}^1(\bb{C}_v)$ where $\bb{C}_v$ is the completion of the algebraic closure of $K_v$. The logarithmic chordal distance $\lambda_v$ is defined to be $-\log \rho_v$. If $v$ is non-archimedean, the function $\lambda_v$ is non-negative. 
\par 
Given $x,y \in \bb{P}^1(K)$ and a finite set of places $S$ of $M_K$ containing all archimedean places, we say that $x$ is $S$-integral relative to $y$ if 
$$\lambda_v(x,y) = 0 \text{ for all } v \not \in S.$$
Intuitively, $\lambda_v(x,y) > 0$ if and only if $x$ and $y$ reduce to the same residue class mod $v$ of $\bb{P}^1$. Hence being $S$-integral means that their residue class are distinct for all $v \not \in S$.

We will need to apply the following quantitative version of Roth's theorem. This appears as Theorem 10 in \cite{HS11} and is originally due to Gross \cite{Gro90}, building on the work of Bombieri--Van-Poorten \cite{BvP88}. We state the version in \cite{Gro90} as we require explicit constants.

\begin{theorem} \label{QuantitativeRoth}
Let $K$ be a number field and $v$ a place of $K$. Let $\beta \in \ovl{K}$ be an element of degree $r > 1$ over $K$. Then there are most $4c_1$ elements $x \in K$ satisfying both
$$N_v \log |x - \beta|^{-1}_v \geq \frac{5}{2} h(x) \text{ and } h(x) \geq c_2 \max\{ h(\beta),1\}$$
where one can take
$$c_1 = (2304  \log r)^3 \text{ and } c_2 = 28 (4608 \log r + 1)!$$
\end{theorem}

\begin{proof}
This is \cite[Corollary]{Gro90} with the same $c_2$, but a different $c_1$. Let $c_1$ be as in \cite[Corollary]{Gro90}. Using the fact that $\log 5rn \leq n-1$ where $n = \lfloor 2304 \log r \rfloor + 1$ as defined in \cite{Gro90},  we have the following inequality
$$c_1 \leq 2304 \log r + 8.5 (2304 \log r) \left( \log (2304 \log r) + \log ((4608 \log r)!) \right)$$ 
$$\leq 9 (2304 \log r) \left(  (4608 \log r) \log (4608 \log r) \right) \leq (2304 \log r)^3$$
and so we may replace $c_1$ with $(2304 \log r)^3$.
\end{proof}

\section{Reduction Steps} \label{sec: Reduction}
The aim of this section is to perform some reductions on Theorem \ref{IntroTheorem3} so that we may assume some advantageous hypotheses on $\vphi,\alpha$ and $\beta$. By iterating $\alpha$ by $\log^+_d(\frac{h(\vphi)+1}{\h_{\vphi}(\alpha)})$ times, we obtain that $\h_{\vphi}(\alpha) \geq h(\vphi)+1$. Hence Theorem \ref{IntroTheorem3} reduces to the following.

\begin{theorem} \label{TheoremReduced1}
Keep the notations of Theorem \ref{IntroTheorem1}. Assume that $\h_{\vphi}(\alpha) \geq h(\vphi)+1$. Then there exists $c_2 = c_2(d)$ such that 
$$|\{ n \geq 0 \mid \vphi^n(\alpha) \text{ is S-integral relative to } \beta\}| \leq c_2 |S| (\log |S|+1)^5.$$
\end{theorem}

We now show that for Theorem \ref{TheoremReduced1}, we may further make the following assumptions:

\begin{enumerate}
\item We may assume that $h(\alpha) \geq h(\vphi)+1$.

\item We may assume that 
$$\left(1+ \frac{1}{100} \right) d^n \h_{\vphi}(\alpha) \geq h(\vphi^n(\alpha)) \geq \left(1 - \frac{1}{100} \right) d^n \h_{\vphi}(\alpha)$$
for all $n \geq 1$. 

\item We may assume that $h(\alpha) \geq h(\beta)$.

\item We may assume that $\infty \not \in \{\alpha,\vphi(\alpha),\ldots\}$. 

\item We may assume that $\infty \not = \beta$ and $\infty \not \in \vphi^{-i}(\beta)$ for all $1 \leq i \leq N$ where $d^{N-1} \leq 56 \cdot 2^{2d-2} |S|$.
\end{enumerate}

Assumptions (4) and (5) are for our convenience to assume that $\infty$ never appears as one of the points that we consider. This is because we will work with the Euclidean distance $\log^+|x-y|_v$ instead of the chordal distance $\lambda_v(x,y)$ and this is only defined for $x,y \in \bb{A}^1$. 
\par 
Assumptions (1) and (2) are satisfied when $\h_{\vphi}(\alpha)$ is sufficiently large compared to $h(\vphi)$ and is easy to prove by replacing $\alpha$ with an iterate. Assumption (3) is more subtle and requires some work. Note that the upper bound in Hsia--Silverman \cite[Equation (3)]{HS11} has $\h_{\vphi}(\beta)$ appearing in the numerator. Assumption (3) is what allows us to remove this term from our bound. 
\par 
Before we begin our reductions, we first state some basic elementary height inequalities. 

\begin{lemma} \label{Heights1}
Let $K$ be a number field and let $x,y \in K$. Then 

\begin{enumerate} 
\item $h(x+y) \leq h(x) + h(y) + \log 2,$
\item $h(x-y) \geq |h(x) - h(y)| - \log 2,$
\item $h(xy) \leq h(x) + h(y).$

\end{enumerate}
\end{lemma}

\begin{proof}
The first inequality follows from Proposition 5(b) of \cite{HS11} as we may view the height of the points $x,y$ as the height of the linear polynomials $z+x$ and $z+y$. The second follows from the first by rearranging it to
$h(x) \geq h(x+y) - h(y) - \log 2$ and the third inequality follows from $\log^+|xy|_v \leq \log^+|x|_v + \log^+|y|_v$.
\end{proof}

Next is a bound between the Weil height and the canonical height of a point $x$ in terms of $h(\vphi)$ and the degree $d$. 

\begin{proposition} \label{Heights2}
Let $\vphi: \bb{P}^1 \to \bb{P}^1$ be a rational map defined over a number field $K$ of degree $d \geq 2$. Then there exists $C = C(d) > 0$ such that
\begin{equation} \label{eq:Heights2}
|\h_{\vphi}(x) - h(x)| \leq   C (h(\vphi)+1).
\end{equation}
for all $x \in \bb{P}^1(\ovl{K})$.
\end{proposition}

\begin{proof}
This is standard. For example, see \cite[Theorem 3.11]{Sil07} for the general case on $\bb{P}^N$. The interested reader may also consult \cite{Ing22} for the higher dimensional case of hypersurfaces on $\bb{P}^N$.
\end{proof}

Now let's first show that we may make assumptions (4) and (5). 

\begin{proposition} \label{Assumption1}
In proving Theorem \ref{TheoremReduced1}, we may assume that $\infty \not \in \{\alpha,\vphi(\alpha),\ldots\}$. 
\end{proposition}

\begin{proof}
Observe that $h(\infty) = 0$. Hence by Proposition \ref{Heights2}, we have 
$$\h_{\vphi}(\infty) \leq C (h(\vphi)+1)$$
for $C = C(d) > 0$. 
Since $\h_{\vphi}(\alpha) \geq h(\vphi)+1$, we may iterate it $\log_d(C) + 2$ many times and using this iterate as our new $\alpha$, we may assume that 
$$\h_{\vphi}(\alpha) > \h_{\vphi}(\infty).$$
Hence $\infty \not \in \{\alpha,\vphi(\alpha),\ldots\}$ as desired. 
\end{proof}

\begin{proposition} \label{Assumption2} 
Let $N$ be the largest positive integer such that 
$$d^{N-1} \leq 56 \cdot 2^{2d-2} |S|.$$
Then we may assume that $\infty \not = \beta$ and  $\infty \not \in \vphi^{-i}(\beta)$ for all $1 \leq i \leq N$.  
\end{proposition}

\begin{proof}
Let $M$ be a positive integer that is not $\beta$ and also not an element of $\vphi^{-i}(\beta)$ for $1 \leq i \leq N$. There are at most $1 + d + \cdots + d^N \leq d^{N+1} - 1$ such elements. Hence we may choose 
$$M \leq d^{N+1} \leq 56d^2 \cdot 2^{2d-2} |S|$$
satisfying this. Let $\psi = \frac{1}{x-M}$. We replace $\vphi$ with $\vphi' = \psi \circ \vphi \circ \psi^{-1}$, $\alpha$ with $\psi(\alpha)$ and $\beta$ with $\psi(\beta)$. This swaps the role of $\infty$ with $M$ and in particular $\infty \not = \beta$ and $\infty \not \in \vphi'^{-i}(\beta)$ for all $1 \leq i \leq N$. Since $\frac{1}{x-M}$ has good reduction for all non-archimedean places, the property of being $S$-integral is preserved and hence the number of $S$-integral points in the orbit of $(\vphi')^n(\psi(\alpha))$ relative to $\psi(\beta)$ remains the same as $\vphi^n(\alpha)$ and $\beta$. 
\par 
The only thing left to check is that $\h_{\vphi'}(\psi(\alpha)) \geq h(\vphi')+1$. The height of the map $\psi = \frac{1}{x-M}$ is $\log M$ and by Proposition 5 of \cite{HS11}, the height of $\vphi'$ can be bounded by 
$$h(\vphi') \leq h(\vphi) + 2d (\log M + \log 8) \leq h(\vphi) + 2d ( (2d-2) + \log |S| + 2 \log d + \log 56 + \log 8).$$
Since $\h_{\vphi'}(\psi(\alpha)) = \h_{\vphi}(\alpha)$, we can now replace $\alpha$ with $\vphi^{n + \log |S|}(\alpha)$ for some $n = n(d) > 0$ depending only on $d$ and obtain that 
$$\h_{\vphi}(\alpha) \geq h(\vphi') + 1$$
as desired. 
\end{proof}

We now handle assumptions (1) and (2). 

\begin{proposition} \label{Assumption3}
We may assume that $h(\alpha) \geq h(\vphi) + 1$. 
\end{proposition}

\begin{proof}
Again, let $C = C(d) > 0$ be the constant in Proposition \ref{Heights2} so that 
$$|\h_{\vphi}(x) - h(x)| \leq C(h(\vphi)+1).$$
Since $\h_{\vphi}(\alpha) \geq h(\vphi)+1$, we can replace $\alpha$ with $\vphi^{n}(\alpha)$ with $n > 0$ depending only on $d$ so that 
$$\h_{\vphi}(\alpha) \geq (C+1)(h(\vphi)+1).$$
Then
$$h(\alpha) \geq h(\vphi)+1$$
as desired.
\end{proof}

\begin{proposition} \label{Assumption4}
We may assume that 
$$\left(1+ \frac{1}{100} \right) d^n \h_{\vphi}(\alpha) \geq h(\vphi^n(\alpha)) \geq \left(1 - \frac{1}{100} \right) d^n \h_{\vphi}(\alpha).$$
\end{proposition}

\begin{proof}
This is similar to Proposition \ref{Assumption3}. If $C = C(d) > 0$ is the constant so that 
$$|\h_{\vphi}(x) - h(x)| \leq C(h(\vphi)+1),$$
then as $\h_{\vphi}(\vphi^n(\alpha)) = d^n \h_{\vphi}(\alpha)$,
$$|d^n \h_{\vphi}(\alpha) - h(\vphi^n(\alpha))| \leq C h(\vphi)+1.$$
Hence if 
$$\frac{1}{100} d^n \h_{\vphi}(\alpha) \geq C h(\vphi)+1$$
holds for all $n \geq 0$, then we are done. Since $\h_{\vphi}(\alpha) \geq h(\vphi)+1$, this can be assumed to hold by replacing $\alpha$ with $\vphi^k(\alpha)$ where $k > 0$ is some large constant depending only on $d$.
\end{proof}

We finally handle assumption (3). As noted previously this step is not as straightforward as the previous reductions. The main idea is that if $h(\beta) >> h(\vphi^n(\alpha))$ for at least $|S|+1$ many $n$'s where $\vphi^n(\alpha)$ is $S$-integral relative to $\beta$, then we may find two of them such that there is a $v \in S$ for which both $|\vphi^{n_1}(\alpha) - \beta|_v$ and $|\vphi^{n_2}(\alpha) - \beta|_v$ are extremely small. This implies that $|\vphi^{n_1}(\alpha) - \vphi^{n_2}(\alpha)|_v$ is very small but we may bound this using $h(\vphi^{n_i}(\alpha))$ which gives rise to a contradiction.

\begin{proposition} \label{Assumption5}
We may assume that $h(\alpha) \geq h(\beta)$. 
\end{proposition}

\begin{proof}
Let $N_S$ be the set of $n$'s for which $\vphi^n(\alpha)$ is $S$-integral relative to $\beta$. For each $n \in N_S$, we have 
$$\sum_{v \in S} N_v \log^+|\vphi^n(\alpha) - \beta|^{-1}_v = \sum_{v \in M_K} N_v \log^+|\vphi^n(\alpha) - \beta|_v^{-1} = h(\vphi^n(\alpha) - \beta).$$
In particular, there exists $v \in S$ such that 
\begin{equation} \label{eq: AssumptionBound1}
N_v \log^+|\vphi^n(\alpha) - \beta|^{-1}_v \geq \frac{1}{|S|} h(\vphi^n(\alpha) - \beta) \geq \frac{1}{|S|} \left( h(\beta) - h(\vphi^n(\alpha)) - \log 2 \right).
\end{equation}
Now if $|N_S| \geq |S|+1$, which we may assume so as otherwise Theorem \ref{TheoremReduced1} is true. Then we may find $n_1,n_2 \in N_S$ among the first $|S|+1$ many values for which \eqref{eq: AssumptionBound1} holds for the same $v \in S$. Now since 
$$|\vphi^{n_1}(\alpha) - \vphi^{n_2}(\alpha)|_v \leq 2 \max\{ |\vphi^{n_1}(\alpha) - \beta|_v , |\vphi^{n_2}(\alpha) - \beta|_v\},$$
it follows that 
$$\log^+|\vphi^{n_1}(\alpha) - \vphi^{n_2}(\alpha)|^{-1}_v + \log 2 \geq \min\{\log^+|\vphi^{n_1}(\alpha) - \beta|^{-1}_v, \log^+|\vphi^{n_2}(\alpha) - \beta|^{-1}_v\}.$$
Applying \eqref{eq: AssumptionBound1}, we obtain 
$$N_v \log^+|\vphi^{n_1}(\alpha) - \vphi^{n_2}(\alpha)|_v^{-1} \geq \frac{1}{|S|} (h(\beta) - h(\vphi^{n_1}(\alpha)) - h(\vphi^{n_2}(\alpha)) - 2 \log 2.$$
If we assume that $n_2 \geq n_1$, then by assumption (2) we certainly have
$$h(\vphi^{n_i}(\alpha)) \leq 2 d^{n_i} \h_{\vphi}(\alpha)$$
and so we have 
$$N_v \log^+|\vphi^{n_1}(\alpha) - \vphi^{n_2}(\alpha)|^{-1}_v \geq \frac{1}{|S|} (h(\beta) - 4 d^{n_2} \h_{\vphi}(\alpha)) - 2 \log 2.$$
But we have the upper bound 
$$N_v \log^+|\vphi^{n_1}(\alpha) - \vphi^{n_2}(\alpha)|_v^{-1} \leq h(\vphi^{n_1}(\alpha) - \vphi^{n_2}(\alpha)) $$
$$\leq h(\vphi^{n_1}(\alpha)) + h(\vphi^{n_2}(\alpha)) + \log 2 \leq 4 d^{n_2} \h_{\vphi}(\alpha) + \log 2.$$
Thus we have 
$$4d^{n_2} \h_{\vphi}(\alpha) + \log 2 \geq \frac{1}{|S|} (h(\beta) - 4d^{n_2} \h_{\vphi}(\alpha)) - 2 \log 2.$$
Since $\h_{\vphi}(\alpha) \geq h(\vphi) + 1$, this implies that replacing $\alpha$ with $\vphi^{n_2}(\alpha)$ and then further iterating it $n + \log |S|$ many times for some $n = n(d) > 0$ depending on $d$ allows us to assume that 
$$\h_{\vphi}(\alpha) \geq  h(\beta) + C(h(\vphi)+1) \implies h(\alpha) \geq h(\beta),$$
where $C$ is the constant again from Proposition \ref{Heights2}. When replacing $\alpha$ with $\vphi^{n_2}(\alpha)$, we lose at most $|S|+1$ many elements in $N_S$ which we can absorb to $c |S|( \log |S|+1)^5$ in Theorem \ref{TheoremReduced1}. Thus we may assume $h(\alpha) \geq h(\beta)$ as desired.
\end{proof}

\section{The non-superattracting case} \label{sec: NonSuperAttracting}
We now prove Theorem \ref{TheoremReduced1} in the case where $\beta$ is not superattracting. Fix a place $v \in M_K$. Recall that we have a logarithmic chordal distance 
$$\lambda_v(x,y) = - \log \rho_v(x,y)$$
where
$$\rho_v(x,y) = \frac{|x_0y_1 - x_1y_0|_v}{\max\{|x_0|_v,|x_1|_v\} \max\{|y_0|_v,|y_1|_v\}}.$$
We record a simple inequality relating $\lambda_v(x,y)$ and $\log^+|x-y|^{-1}$, where $\log^+|z|^{-1} = \max\{0, -\log|z|\}$.

\begin{proposition} \label{Distance1}
There exists a constant $C > 0$ such that for all $x,y \in \bb{A}^1(\bb{C}_v)$, we have
$$\left| \lambda_v(x,y) - \log^+|x-y|_v^{-1} \right| \leq C (\log^+|y|_v + \ell_v)$$
where $\ell_v = 1$ for archimedean $v$ and $0$ for non-archimedean $v$.
\end{proposition}

\begin{proof}
We have 
$$\lambda_v(x,y) = -\log |x-y|_v + \log^+|x|_v + \log^+|y|_v.$$
If $|x-y|_v \leq 1$, then $\log^+|x-y|_v^{-1} = - \log |x-y|_v$ and the difference is $\log^+|x|_v + \log^+|y|_v$ which may be bounded by $2 \log^+|y| + 2 \ell_v$ as $|x-y|_v \leq 1$. If $|x-y|_v \geq 1$, then $\log^+|x-y|_v^{-1} = 0$ and we are to bound $-\log |x-y|_v + \log^+|x|_v + \log^+|y|_v$. Note that it suffices to bound it from above as $\lambda_v(x,y)$ is uniformly bounded from below, with a bound of $0$ for non-archimedean $v$. 
\par 
If $|x-y|_v \geq \frac{1}{2} |x|_v$, then we may bound the expression from above by $\log^+|y|_v + 1$. Otherwise if $|x|_v \geq 2|x-y|_v$, we must have $|y|_v \geq \frac{1}{2}|x|_v$ and so we may bound the expression by $C(\log^+|y|_v + 1)$. This handles the case when $v$ is archimedean.
\par 
If $v$ is non-archimedean, then we either have $|x-y|_v = \max\{|x|_v,|y|_v\}$ or $|x| = |y|_v$. If $|x|)v = |y|_v$, then we can bound the expression by $C \log^+|y|_v$ and similarly if $|x|_v < |y|_v$. Otherwise $|x|_v > |y|_v$ and so $\log |x-y|_v = |x|_v$. Hence the expression reduces to $\log^+|y|_v$ which finishes the proof. 
\end{proof}

We now have the following version of the inverse function theorem by Hsia--Silverman \cite[Proposition 7]{HS11} where we use $\log^+|x-y|_v^{-1}$ instead of $\lambda_v(x,y)$. For any point $x \in \bb{P}^1(\ovl{K})$, we let $e_x(\vphi)$ denote the ramification degree of $\vphi$ at $x$, i.e. the multiplicity of $x$ as root of $\vphi(z) = \vphi(x)$. This was proven using the tools from \cite{Sil87}, see also the more recent work of Matsuzawa--Silverman \cite{MS21}. 

\begin{proposition} \label{HsiaSilvermanImplicit1}
There exists $C > 0$ depending only on the degree $d$ such that for any $x,\beta \in \bb{A}^1(K)$ with $\infty \not \in \vphi^{-1}(\beta)$ and $\vphi(x) \not = \infty$, and any finite set of places $S \subseteq M_K$, we have 
$$\sum_{v \in S} \max_{\beta_i \in \vphi^{-1}(\beta)} e_{\beta_i}(\vphi) n_v \log^+|x - \beta_i|^{-1}_v \geq \sum_{v \in S} N_v \log^+|\vphi(x)-\beta|^{-1}_v - C(h(\beta) + h(\vphi)+1)$$
\end{proposition}

\begin{proof}
By \cite[Proposition 7]{HS11}, the inequality holds with $\lambda_v(x,\beta)$ replacing $\log^+|x-\beta|^{-1}_v$. As $e_{\beta_i}(\vphi) \leq 2d-2$, Proposition \ref{Distance1} allows us to replace $\lambda_v(x,\beta)$ with $\log^+|x-\beta|^{-1}_v$ and incurring an error of 
$$(2d-2) C' \sum_{v \in S} N_v(\log^+|\beta|_v+\ell_v)$$
for some uniform $C' = C'(d) > 0$ depending only on the degree $d$. We may then absorb the error into $C(h(\beta)+1)$ as desired.
\end{proof}

We may now repeatedly iterate Proposition \ref{HsiaSilvermanImplicit1} to obtain an error bound for $\vphi^n$ depending on a fixed power of $d$.

\begin{proposition} \label{HsiaSilvermanImplicit2}
There exists $c = c(d) > 0$ depending only on the degree $d$ such that for any $x,\beta \in \bb{A}^1(K)$ with $\infty \not \in \vphi^{-i}(\beta)$ and $\vphi^i(x) \not = \infty$ for $1 \leq i \leq n$, and any finite set of places $S \subseteq M_K$, we have 
$$\sum_{v \in S} \max_{\beta_i \in \vphi^{-n}(\beta)} e_{\beta_i}(\vphi^n) N_v \log^+|x - \beta_i|^{-1}_v \geq \sum_{v \in S} n_v \log^+|\vphi^n(x)-\beta|^{-1}_v - d^{cn}(h(\beta) + h(\vphi)+1)$$
\end{proposition}

\begin{proof}
We prove this by induction, where the existence of such a $c = c(d) > 0$ for case of $n = 1$ follows from Proposition \ref{HsiaSilvermanImplicit1}. We now fix such a $c$ that also satisfies 
$$d^{ck} + d^{k+c} \leq d^{(k+1)c}$$
for all $k \geq 1$. This can be done by making $c$ large enough. 
By Proposition \ref{Heights2}, we may work with the quantity $\h_{\vphi}(\beta)$ instead of $h(\beta)$ for the proof of Proposition \ref{HsiaSilvermanImplicit2}. Assume that we have proven for $n = k$. We thus obtain 
\begin{equation} \label{eq: Implicit1}
\sum_{v \in S} \max_{\beta_i \in \vphi^{-k}(\beta)} e_{\beta_i}(\vphi^k) N_v \log^+|\vphi(x) - \beta_i|_v^{-1}
\end{equation}
$$
\geq \sum_{v \in S} n_v \log^+|\vphi^{k+1}(x) - \beta|_v^{-1} - d^{ck}(\h_{\vphi}(\beta) + h(\vphi)+1).
$$
Now for each $\beta_i \in \vphi^{-k}(\beta)$, as $\h_{\vphi}(\beta_i) \leq \h_{\vphi}(\beta)$, we  can apply the $n = 1$ case to it to obtain 
\begin{equation} \label{eq: Implicit2}
\sum_{v \in S} \max_{\beta_j \in \vphi^{-1}(\beta_i)} e_{\beta_j}(\vphi) N_v \log^+|x - \beta_j|_v^{-1} \geq \sum_{v \in S} N_v \log^+|\vphi(x) - \beta_i|_v^{-1} - d^c (\h_{\vphi}(\beta) + h(\vphi)+1).
\end{equation}
Now observe that if $\vphi(\beta_j) = \beta_i$, then $e_{\beta_j}(\vphi) e_{\beta_i}(\vphi^k) = e_{\beta_j}(\vphi^{k+1})$. We now apply \eqref{eq: Implicit2} to \eqref{eq: Implicit1} on each $\beta_i \in \vphi^{-k}(\beta)$, for which there are at most $d^k$ many of them. We thus obtain 
$$\sum_{v \in S} \max_{\beta_j \in \vphi^{-(k+1)}(\beta)} e_{\beta_j}(\vphi^{k+1}) N_v \log^+|x-\beta_j|^{-1}_v $$
$$\geq \sum_{v \in S} n_v \log^+|\vphi^{k+1}(x) - \beta|_v^{-1} - (d^{ck} + d^{k+c})(\h_{\vphi}(\beta) + h(\vphi)+1).$$
Since we assumed $c$ was large enough so that 
$$d^{ck} + d^{k+c} \leq d^{(k+1)c},$$
the last expression is 
$$\geq \sum_{v \in S} N_V \log^+|\vphi^{k+1}(x) - \beta|_v^{-1} - d^{c(k+1)} (\h_{\vphi}(\beta) + h(\vphi)+1)$$
as desired, completing the induction. 
\end{proof}

We also need the following estimate which deals with how far $\vphi(z)$ and $\vphi(\beta)$ be apart compared to $z$ and $\beta$. This reduces to bounding the spherical Lipschitz constant of $\vphi$, which we will use results due to Rumely--Winburn \cite{RL15} and Gauthier--Okuyama--Vigny \cite{GOV20}. 

\begin{proposition} \label{Distance2}
There exists $C_1,C_2 > 0$ depending only on $d$ such that for $z,\beta \not = \infty$ with $\vphi(z), \vphi(\beta) \not = \infty$, if 
$$\log^+|z-\beta|^{-1}_v \geq C_1[K:\bb{Q}](h(\vphi) + h(\beta)+1),$$
then 
$$\log^+|z-\beta|^{-1}_v -  \log^+|\vphi(z) - \vphi(\beta)|^{-1}_v  \leq C_2[K:\bb{Q}](h(\vphi) + h(\beta)+1).$$
\end{proposition}

\begin{proof}
By Proposition \ref{Distance1}, it suffices to prove the existence of a constant $C = C(d) > 0$ such that 
$$\lambda_v(z,\beta) - \lambda_v(\vphi(z),\vphi(\beta)) \geq -C[K:\bb{Q}](h(\vphi)+1)$$
 as $\log^+|\beta|_v \leq [K:\bb{Q}]h(\beta)$. Bounding $\lambda_v(z,\beta) - \lambda_v(z,\vphi(\beta))$
from below is equivalent to giving a Lipschitz constant for the chordal metric $\rho_v$ as $\lambda_v = - \log \rho_v$. We apply Lemma 2.5 of \cite{GOV20} to see that there is a constant $C' = C'(d) > 0$ depending only on $d$ such that 
$e^{C[K:\bb{Q}](h(f)+1)}$ is a Lipschitz constant for $f$. This is standard by picking a lift $F$ of $f$ so that $|F| = 1$. Then $-\log |\Res(F)|_v$ can be bounded from above in terms of $c [K:\bb{Q}]h(f)$ where $c$ depends on the degree $d$. Hence we obtain 
$$\lambda_v(z,\beta) - \lambda_v(z,\vphi(\beta)) \geq -C' [K:\bb{Q}](h(f)+1)$$
which gives us what we want.
\end{proof}

We are now ready to begin the proof of Theorem \ref{TheoremReduced1} assuming that $\beta$ is not superattracting. We set $N$ to be the largest positive integer such that 
$$d^{N-1} \leq 56 \cdot 2^{2d-2} |S|.$$
First, because $\beta$ is not superattracting, by \cite[Lemma 4.3]{Sil93}, we must have 
$$e_{\beta'}(\vphi^n) \leq 2^{2d-2}$$
for any $\beta' \in \vphi^{-n}(\beta)$ and for all $n \geq 1$. This bounds $e_{\beta'}(\vphi^n)$ uniformly from above in Proposition \ref{HsiaSilvermanImplicit2}. We emphasize that this is the only point where we use that $\beta$ is non-superattracting. 
\par 
Next, by Section \ref{sec: Reduction}, we can assume that the following assumptions hold:

\begin{enumerate}
\item We may assume that $\h_{\vphi}(\alpha) \geq h(\vphi) + 1$.

\item We may assume that 
$$\left(1+ \frac{1}{100} \right) d^n \h_{\vphi}(\alpha) \geq h(\vphi^n(\alpha)) \geq \left(1 - \frac{1}{100} \right) d^n \h_{\vphi}(\alpha)$$
for all $n \geq 1$. 

\item We may assume that $h(\alpha) \geq h(\beta)$.

\item We may assume that $\infty \not \in \{\alpha,\vphi(\alpha),\ldots\}$. 

\item We may assume that $\infty \not \in \vphi^{-i}(\beta)$ for all $1 \leq i \leq N$ where $N$ is the largest positive integer such that $d^{N-1} \leq 56 \cdot 2^{2d-2} |S|$.
\end{enumerate}

We first show the following. This is similar to (22) in \cite{HS11}, except that we look at the $N^{th}$ preimage where $N$ depends on $|S|$ and so we have to be more careful with our bounds. 

\begin{lemma} \label{SIntegralBound1}
Assume that $\vphi^n(\alpha)$ is $S$-integral relative to $\beta$. There exists a constant $n_0 = n_0(d) > 0$ such that if $n \geq n_0 \log |S|$, we have the lower bound
$$\sum_{v \in S} \sum_{\beta_i \in \vphi^{-N}(\beta)} N_v \log^+|\vphi^{n-N}(\alpha)-\beta_i|_v^{-1} \geq \frac{d^{n}}{3 \cdot 2^{2d-2}} h(\alpha).$$
\end{lemma}

\begin{proof}
We apply Proposition \ref{HsiaSilvermanImplicit2} and obtain a constant $C_1 > 0$, depending only on $d$, such that
$$2^{2d-2} \sum_{v \in S} \sum_{\beta_i \in \vphi^{-N}(\beta)} N_v\log^+|\vphi^{n-N}(\alpha)-\beta_i|_v^{-1}$$
$$\geq \sum_{v \in S} N_v \log^+|\vphi^n(\alpha)-\beta|_v^{-1} -  d^{C_1 N } \left( h(\vphi) + h(\beta) + 1 \right)$$
$$= h(\vphi^n(\alpha) - \beta) - d^{C_1 N}(h(\vphi) + h(\beta)+1).$$
By assumption, we may have $h(\vphi^n(\alpha)) \geq \frac{2d^n}{3}h(\alpha)$ and $h(\alpha) \geq h(\vphi)+1,h(\beta)$. For $n$ large enough say $\geq 100$, we certainly have 
$$h(\vphi^n(\alpha) - \beta) \geq h(\vphi^n(\alpha)) - h(\beta) - \log 2 \geq \frac{d^n}{2} h(\alpha).$$
We may then choose $n_0$ large enough, depending only on $d$, so that 
$$\frac{d^{n_0 \log |S|}}{6} h(\alpha) \geq d^{C_1 N}(h(\vphi)+h(\beta)+1)$$
as $d^{N-1} \leq 56 \cdot 2^{2d-2} |S|$. Thus 
$$2^{2d-2} \sum_{v \in S} \sum_{\beta_i \in \vphi^{-N}(\beta)} N_v \log^+|\vphi^{n-N}(\alpha) - \beta_i|^{-1}_v \geq \frac{d^n}{2} h(\alpha) - \frac{d^n}{6}h(\alpha) = \frac{d^n}{3}h(\alpha).$$
Dividing by $2^{2d-2}$ gives us what we want.
\end{proof}

For each $v \in S$, we let $\beta_{v,n}$ be the element of $\vphi^{-N}(\beta)$ such that the quantity $\log^+|\vphi^{n-N}(\alpha)-\beta_i|^{-1}_v$ obtains its maximum over all $\beta_i \in \vphi^{-N}(\beta)$. We next show that we can obtain a similar inequality just by using the $\beta_{v,n}$'s. 

\begin{lemma} \label{SIntegralBound2}
Assume that $\vphi^n(\alpha)$ is $S$-integral relative to $\beta$. There exists a constant $n_0 = n_0(d) > 0$ such that if $n \geq n_0 \log |S|$, we have
$$\sum_{v \in S} N_v \log^+|\vphi^{n-N}(\alpha) - \beta_{v,n}|_v^{-1} \geq \frac{d^n}{4 \cdot 2^{2d-2}} h(\alpha).$$
\end{lemma}

\begin{proof}
For any $\beta_i \in \vphi^{-N}(\beta)$ that is not $\beta_{v,n}$, we have 
$$|\vphi^{n-N}(\alpha) - \beta_i|_v + |\vphi^{n-N}(\alpha) - \beta_{v,n}|_v \geq |\beta_i - \beta_{v,n}|_v \implies |\vphi^{n-N}(\alpha) - \beta_i|_v \geq \frac{1}{2} |\beta_i - \beta_{v,n}|_v$$
due to the maximality of $\beta_{v,n}$. Thus 
$$\log^+|\vphi^{n-N}(\alpha)-\beta_i|^{-1}_v \leq \log 2 + \log^+|\beta_i - \beta_{v,n}|_v^{-1}.$$
We now sum over $v \in S$ to get
$$\sum_{v \in S} N_v \log^+|\vphi^{n-N}(\alpha) - \beta_i|_v^{-1} \leq |S| \log 2 + \sum_{v \in S} N_v \log^+|\beta_i - \beta_{v,n}|_v^{-1}.$$
Now each element of $\vphi^{-N}(\beta)$ is of degree at most $d^{N} \leq 56 d \cdot 2^{2d-2}|S|$ over $K$ and so $\beta_i - \beta_{v,n}$ is of degree at most $(56 d \cdot 2^{2d-2}|S|)^2$. In particular, 
$$N_v \log^+|\beta_i - \beta_{v,n}|_v^{-1} \leq (56 d \cdot 2^{2d-2} |S|)^2 h(\beta_i - \beta_{v,n}).$$
Thus we obtain
$$\sum_{v \in S} \sum_{\substack{\beta_i \in \vphi^{-N}(\beta) \\ \beta_i \not = \beta_{v,n}}} N_v \log^+|\vphi^{n-N}(\alpha) - \beta_i|^{-1} \leq d^N |S| \log 2 + d^N |S| (56 d \cdot 2^{2d-2} |S|)^2 h(\beta - \beta_{v,n}).$$
Now we have 
$\h_{\vphi}(\beta_{v,n}) \leq \h_{\vphi}(\beta)$ and so if $C_1$ is the constant from Proposition \ref{Heights2}, we have
$$h(\beta_{v,n}) \leq h(\beta) + 2C_1(h(\vphi) + 1).$$
Thus 
$$h(\beta - \beta_{v,n}) \leq h(\beta) + h(\beta_{v,n}) + \log 2 \leq 2 h(\beta) + 2C_2(h(\vphi)+1)$$
for some $C_2$ depending only on $d$. Now by Lemma \ref{SIntegralBound1}, if $n \geq n_0 \log |S|$ for $n_0$ large enough depending on $d$, we get
$$\sum_{v \in S} N_v \log^+|\vphi^{n-N}(\alpha) - \beta_{v,n}|_v^{-1} \geq \frac{d^n}{2^{2d-2} \cdot 3} h(\alpha) - \sum_{v \in S} \sum_{\substack{\beta_i \in \vphi^{-N}(\beta) \\ \beta_i \not = \beta_{v,n}}} N_v \log^+|\vphi^{n-N}(\alpha) - \beta_i|_v^{-1}$$
$$\geq \frac{d^n}{3 \cdot 2^{2d-2} } h(\alpha) - d^N |S| \log 2 - d^N |S|(56 d \cdot 2^{2d-2} |S|)^2 (2 h(\beta) + 2C_2 h(\vphi)+1).$$ 
We can make this $\geq \frac{d^n}{4 \cdot 2^{2d-2}} h(\alpha)$
if $n \geq n_0 \log |S|$ where we take $n_0$ large enough depending only on $d$ as desired since $h(\alpha) \geq h(\vphi)+1,h(\beta)$. This finishes the proof of our lemma.
\end{proof}

 If one is happy with a polynomial bound in $|S|$, then Lemma \ref{SIntegralBound2} only with Roth's theorem is enough. To see this, note that by choice of our $N$, we have 
 $$\frac{d^n}{4 \cdot 2^{2d-2}} \geq 3|S| d^{n-N}.$$ 
 Then we always have a place $v \in S$ and $\beta_i \in \vphi^{-N}(\beta)$ for which 
 \begin{equation} \label{eq: SIntegralBound1} 
 N_v \log^+|\vphi^{n-N}(\alpha) - \beta_i| \geq \frac{5}{2} h(\vphi^{n-N}(\alpha)).
 \end{equation}
As each $\beta_i$ is of degree at most $d^N$, there can be at most $N \log d$ many distinct $\vphi^{n-N}(\alpha)$'s satisfying this which is $O(d \log |S|)$. Thus we have at most $O(d \log |S|)$ many possibilities for each pair $(v,\beta_i)$ and there are at most $O(2^{2d-2}|S|^2)$ many such pairs. This gives a polynomial bound as desired.
\par 
To improve this to a quasilinear bound, we have to do some combinatorial analysis. Let's first see why it is reasonable to expect an improvement in the above argument. First, we have 
$$\sum_{v \in S} N_v \log^+|\vphi^{n-N}(\alpha) - \beta_{v,n}|^{-1}_v \geq \frac{d^n}{3} h(\alpha).$$
To obtain \eqref{eq: SIntegralBound1}, we argued that for a single place $v$, we have 
$$N_v \log^+|\vphi^{n-N}(\alpha) - \beta_{v,n}|^{-1}_v \geq \frac{d^n}{3|S|} h(\alpha).$$
If we cannot improve this bound, this implies that this inequality holds for all places $v \in S$ and hence $\vphi^{n-N}(\alpha)$ is a counterexample to Roth's theorem for $|S|$ pairs $(\beta_{v,n},v)$. This would then allow for at most $O(|S|)$ many $\vphi^{n-N}(\alpha)$'s. 
\par 
Going to the other extreme, the inequality may only hold for one place $v$, in which case for simplicity we would have an inequality of the form 
$$N_v \log^+|\vphi^{n-N}(\alpha) - \beta_{v,n}|^{-1}_v \geq \frac{d^n}{4}h(\alpha).$$
Then ignoring errors, we would obtain an inequality of the form  
$$N_v \log^+|\vphi^{n-4}(\alpha) - \vphi^{N-4}(\beta_{v,n})|^{-1}_v \geq \frac{d^n}{5} h(\alpha)$$
and since $\frac{25}{2} \leq d^3$, it follows that $\vphi^{n-4}(\alpha) - \vphi^{N-4}(\beta_{v,n})$ is a counterexample to Roth. But now $\vphi^{N-4}(\beta_{v,n}) \in \vphi^{-4}(\beta)$ of which there is only at most $d^4$ many possibilities. So again we are left with a $O(|S|)$ number of possibilities. 
\par 
Of course the above two cases are the extreme cases and there are many more possibilities. To quantify the different possibilities, we define the following subsets $X_i \subseteq \bb{N}$ as follows. We say that $n \in \bb{N}$ belongs to $X_i$ if 
$$N_v \log^+| \vphi^{n-N}(\alpha) - \beta_{v,n}|_v^{-1} \geq 4 d^{n-N+i} h(\alpha)  \text{ for at least } \frac{d^{N-i-1}}{56N \cdot 2^{2d-2}} \text{ distinct } v  \in S.$$
Then $i = 0$ and $i = N-1$ correspond roughly to the two extreme cases discussed above.

\begin{lemma} \label{SIntegralBound3}
Assume that $\vphi^n(\alpha)$ is $S$-integral relative to $\beta$ and that $n \geq n_0 \log |S| $ for $n_0$ large enough for Lemma \ref{SIntegralBound2} to hold. Then $n \in X_i$ for at least one $i$ with $0 \leq i \leq N-1$. 
\end{lemma}

\begin{proof}
Let's assume otherwise that $n \not \in X_i$ for all $0 \leq i \leq N-1$. For $1 \leq i \leq N$, let $S_i$ be the subset of places of $S$ that satisfy
$$N_v \log^+|\vphi^{n-N}(\alpha) - \beta_{v,n}|^{-1}_v \geq 4 d^{n-N+i} h(\alpha).$$
We have the chain of inclusions $S \supset S_0 \supset S_1 \cdots \supset S_{N-1}$. Then since $n \not \in X_i$, we must have $|S_i| \leq \frac{d^{N-i-1}}{56N \cdot 2^{2d-2}}$. Let $T = S \setminus S_0$. For $v \in T$, we then have the upper bound 
$$N_v \log^+|\vphi^{n-N}(\alpha) - \beta_{v,n}|_v^{-1} \leq 4d^{n-N} h(\alpha).$$
For each $v \in  S_{i} \setminus S_{i+1}$ where $i \geq 0$, we have the upper bound 
$$N_v \log^+|\vphi^{n-N}(\alpha) - \beta_{v,n}|_v^{-1} \leq 4d^{n-N+i+1} h(\alpha).$$
Furthermore, $S_{N-1}$ is empty since
$$|S_{N-1}| \leq \frac{1}{56N \cdot 2^{2d-2}}.$$
For $i \geq 1$, we have $|S_{i} \setminus S_{i+1}| \leq |S_i| \leq \frac{d^{N-i-1}}{56N \cdot 2^{2d-2}}$. Hence 
$$\sum_{v \in S_{i} \setminus S_{i+1}} N_v \log^+|\vphi^{n-N}(\alpha) - \beta_{v,n}|_v^{-1} \leq 4d^{n-N+i+1} h(\alpha) \frac{d^{N-i-1}}{56N \cdot 2^{2d-2}} = \frac{d^n}{14N \cdot2^{2d-2}} h(\alpha).$$
Writing $S$ as $T \bigcup \left(\bigcup_{i=0}^{N-1} (S_{i} \setminus S_{i+1}) \right)$, we have the bound 

$$\sum_{v \in S} N_v \log^+|\vphi^{n-N}(\alpha) - \beta_{v,n}|_v^{-1} \leq \sum_{i=0}^{N-1} \frac{d^n}{14N \cdot2^{2d-2}} h(\alpha) + 4|S| d^{n-N} h(\alpha)$$ 
$$\leq \frac{d^n}{14 \cdot2^{2d-2}}  h(\alpha) + \frac{d^n}{14 \cdot2^{2d-2}} h(\alpha) = \frac{d^n}{7 \cdot 2^{2d-2}} h(\alpha).$$
But this is a contradiction to Lemma \ref{SIntegralBound2}, which implies that $n \in X_i$ for at least one $i$ with $0 \leq i \leq N-1$.
\end{proof}

Before finishing the proof of Theorem \ref{TheoremReduced1}, we need to note the following inequality. 

\begin{proposition} \label{SIntegralBound4}
There exists $n_0 = n_0(d) > 0$ such that if $n \geq n_0 \log |S|$, then 
$$N_v \log^+|\vphi^{n-N}(\alpha) - \beta'|_v^{-1} \geq 4d^{n-N+i}h(\alpha)$$
for $\beta' \in \vphi^{-N}(\beta)$ implies that 
$$N_v \log^+|\vphi^{n-N+i}(\alpha) - \vphi^i(\beta')|_v^{-1} \geq 3 d^{n-N+i} h(\alpha).$$
\end{proposition}

\begin{proof}
We apply Proposition \ref{Distance2}. We have to be careful as we have to deal with $[K(\beta'):\bb{Q}]$ instead of $[K:\bb{Q}]$. For this, we note that $$[K(\beta'):\bb{Q}] \leq d^N \cdot [K:\bb{Q}] \leq 56d \cdot 2^{2d-1} \cdot |S|^2$$
holds for any $\beta' \in \vphi^{-j}(\beta)$ for $1 \leq j \leq N$, as $[K:\bb{Q}] \leq 2|S|$. Now let $C_1,C_2$ be the two constants from Proposition \ref{Distance2}. Then if $n_0$ is taken large enough depending only on $d$, we will have 
$$\frac{1}{N} d^{n-N}h(\alpha) \geq  \max\{C_1,C_2\} [K(\beta'):\bb{Q}](h(\vphi) + h(\beta')+1)$$
for all $\beta' \in \vphi^{-j}(\beta)$, for any $1 \leq j \leq N$. Indeed, this follows from $h(\alpha) \geq h(\vphi)+1, h(\beta)$, our upper bound on $[K(\beta'):\bb{Q}]$, and Proposition \ref{Heights2} along with the fact that $\h_{\vphi}(\beta') \leq \h_{\vphi}(\beta)$. 
\par 
We may now apply Proposition \ref{Distance2} a total of $i$ times to obtain that 
$$N_v \log^+|\vphi^{n-N+i}(\alpha) - \vphi^i(\beta')|_v^{-1} \geq 4d^{n-N+i} h(\alpha) - \frac{i}{N} d^{n-N} h(\alpha) \geq 3d^{n-N+i} h(\alpha)$$
as desired. 
\end{proof}

We can now conclude Theorem \ref{TheoremReduced1} by a simple counting argument. 

\begin{proof}[Proof of Theorem \ref{TheoremReduced1} when $\beta$ is not superattracting]
Let $N_S$ be the set of natural numbers so that $\vphi^n(\alpha)$ is $S$-integral relative to $\beta$. By Lemma \ref{SIntegralBound3}, we may fix $n_0 = n_0(d) > 0$ such that for all $n \geq n_0 \log |S|$, we have that $n \in N_S$ implies $n \in X_i$ for some $0 \leq i \leq N-1$. Assume $n_0$ is also large enough for Proposition \ref{SIntegralBound4} to hold too. By pigeonhole, at least 
$$\frac{|N_S|- n_0\log |S|}{N}$$
many $n \in N_S$ belong to the same $X_i$. For such a $n$, we obtain $\frac{d^{N-i-1}}{56 N \cdot 2^{2d-2}}$ many $v \in S$ such that 
$$N_v \log^+|\vphi^{n-N}(\alpha) - \beta_{v,n}|_v^{-1} \geq 4d^{n-N+i}h(\alpha).$$
Applying Proposition \ref{SIntegralBound4}, we get
$$N_v \log^+|\vphi^{n-N+i}(\alpha) - \vphi^i(\beta_{v,n})|^{-1}_v \geq 3d^{n-N+i}h(\alpha).$$
Since $h(\vphi^{n-N+i}(\alpha)) \leq \frac{6}{5} d^{n-N+i}h(\alpha)$, each pair $(\vphi^{n-N+i}(\alpha), v)$ gives us a counterexample to Roth's theorem for $\vphi^i(\beta_{v,n})$. Since $\vphi^i(\beta_{v,n}) \in \vphi^{-N+i}(\beta)$, there are at most $d^{N-i}$ possibilities. Each such $n$ gives us $\frac{d^{N-i-1}}{56 N \cdot 2^{2d-2}}$ many pairs. 
\par 
Assume that $|N_S| \geq c |S| N^2 + n_0 \log |S|$ for some $c > 0$. Then we have a total of 
$$\frac{c |S| N^2}{N} \cdot \frac{d^{N-i-1}}{56 N \cdot 2^{2d-2}}$$
many pairs. Thus there must be a $\beta' \in \vphi^{-N+i}(\beta)$ and $v \in S$ for which we have at least 
$$\frac{c |S| N^2 d^{N-i-1}}{d^{N-i} (56) 2^{2d-2} |S| N^2} = \frac{c}{56 d 2^{2d-2}}$$
many $n$'s for which 
$$N_v \log^+|\vphi^{n-N+i}(\alpha) - \beta'|_v \geq \frac{5}{2} h(\vphi^{n-N+i}(\alpha)).$$
By Theorem \ref{QuantitativeRoth}, there are at most $O(\log r)^3$ many such $n$'s if 
$$h(\vphi^n(\alpha)) \geq 28(4608 \log r+1)! \max\{1,h(\beta)\}.$$
 Here, we may take $r \leq d^N \leq C' \cdot |S|$ for some $C' = C'(d) > 0$. Then the height inequality is attained if $n \geq |S| + n_0 \log |S|$ for $n_0$ sufficiently large depending on $d$. We can certainly assume this is the case for our $n$'s as we may just replace $\alpha$ with $\vphi^{|S| + n_0 \log |S|}(\alpha)$. 
 Hence if $c \geq C' (56 d 2^{2d-2}) (\log |S|)^3$, we obtain a contradiction. This gives us at most 
$$C' |S| (56 d 2^{2d-2})N^2 (\log |S|)^3$$
many possible $n$'s, which we may bound by $c|S| (\log |S|)^5$ as desired for some $c = c(d) > 0$.
\end{proof}

\section{The superattracting case} \label{sec: Superattracting}
We now handle the case where $\beta$ is superattracting. We are unable to use the same argument as in Section \ref{sec: NonSuperAttracting} since the multiplicity $e_{\beta'}(\vphi^n)$ will grow as $n$ increases, causing our bounds to be weaker than even Hsia--Silverman's \cite{HS11}.
\par 
We instead exploit the fact that $\beta$ is superattracting to bound how close $\vphi^{n}(\alpha)$ can be near $\beta$. Indeed if we assume that $\beta$ is fixed by $\vphi$ with multiplicity $e$ for simplicity, assuming that $\alpha$ was sufficiently close to $\beta$, we must have 
\begin{equation} \label{eq: Superattracting1}
\log^+|\vphi^{n}(\alpha) - \beta|^{-1}_v \approx e^n \log^+|\alpha - \beta|^{-1}_v.
\end{equation}
On the other hand, $h(\vphi^n(\alpha) -\beta)$ grows like $d^n h(\alpha)$ and as $e \leq d-1$, this leads to a contradiction once $n$ is sufficiently large.  
\par 
We remark that although the inequalities involved appear more technical, this case is more elementary as we do not use Roth's theorem and only use height inequalities coming from the product formula.
\par 
We first quantify \eqref{eq: Superattracting1}. Given our rational map $\vphi: \bb{P}^1 \to \bb{P}^1$ of degree $d \geq 2$ over $K$ and an $\beta \in \bb{A}^1(K)$, we give an upper bound on $d_n$ where $d_n$ is the $n$th Taylor series coefficient. 

\begin{proposition} \label{DerivativeBound1}
 Assume that $\beta,\vphi(\beta) \in K$. Write 
 $$\vphi(z) = \vphi(\beta) + d_1 (z-\beta) + d_2 (z-\beta)^2 + \cdots$$
 Then there exists $C = C(d) > 0$ depending only on $d$ such that  
$$h(d_n) \leq Cn(h(\vphi) + h(\beta)+1).$$. 
\end{proposition}

\begin{proof}
We may translate $\beta$ so that $\beta = 0$. Indeed, this increases $h(\vphi)$ by at most $c(h(\beta)+1)$ for $c = c(d) > 0$ depending only on $d$ which is fine for us. Now write 
$$\vphi(z) = \frac{f(z)}{g(z)} = \frac{a_0 + a_1 z + \cdots + a_d z^d}{b_0 + b_1 z + \cdots + b_d z^d}$$
where $b_0 = 1$ since $\vphi(\beta) \not = \infty$. Then $h(a_i),h(b_j) \leq h(\vphi)$ by definition. We now expand 
$$\frac{1}{1 -(- b_1z - \cdots - b_dz^d)} = 1 + (b_1z + \cdots + b_d z^d) + (b_1z + \cdots + b_d z^d)^2 + \cdots = 1 + c_1z + c_2 z^2 + \cdots.$$
For non-archimedean $v$, we may bound 
$$|c_n|_v \leq \prod_{i=1}^{d} \max\{|b_i|_v,1\}^n $$
as $c_n$ is a sum of terms where each term is a product of at most $n$ many $b_i$'s. For archimedean $v$, we have 
$$|c_n|_v \leq  (\sum_{i=1}^{n} 2^i) \prod_{i=1}^{d} \max\{|b_i|_v,1\}^n$$
as it is a sum of at most $\sum_{i=1}^{n} 2^i$ many terms, where each term is a product of at most $n$ many $b_i$'s. 
Hence if $g_n = 1 + c_1 z + \cdots + c_n z^n$, we have 
$$h(g_n) \leq n \sum_{i=1}^{d} h(b_i) + \log( 2^{n+1}) \leq dn h(\vphi) + (n+1) \log 2.$$
Finally, if 
$$(a_0 + a_1 z + \cdots + a_d z^d)(1 + c_1 z + \cdots) = d_0 + d_1 z + \cdots,$$
then applying Proposition 5 of \cite{HS11} to $f(z) = a_0 + a_1z + \cdots + a_d z^d$ and $g(z) = 1 + c_1z + \cdots + c_n z^n$ gives us 
$$h(d_n) \leq h(d_0) + (h(f) + h(g_n) + (d+n+2) \log 2 \leq (dn+2) h(\vphi) + (2n+3 + d)\log 2$$
Here, the quantity $h(d_0)$ comes from the fact that if $Q = e_0 + e_1 z + \cdots + e_n z^n$ satisfies $h(Q) \leq N$, then $h(\frac{e_i}{e_0}) \leq N$ and so $h(e_i) \leq h(e_0) + N$. Then as $d_0 = a_0$, we may bound $h(d_0) \leq h(\vphi)$. This finishes our proof. 
\end{proof}

As in Proposition \ref{DerivativeBound1}, consider the expansion 
$$\vphi(z) = \vphi(\beta) + d_1 (z-\beta) + d_2 (z-\beta)^2 + \cdots$$
where now we have an upper bound on $h(d_n)$ that is linear in $n$. This allows us to bound how small $|z-\beta|$ has to be in order for \eqref{eq: Superattracting1} to hold.

\begin{proposition} \label{DerivativeBound2}
Let $K$ be a number field and $\vphi$ a rational map defined over $K$ of degree $d \geq 2$. Assume that $\beta$ and $\vphi(\beta)$ are both in $K$. Then for any place $v \in M_K$, there exists constants $C_1 = C_1 (d) > 0, C_2 = C_2(d) > 0$, both depending only on the degree $d$, with $C_1 \geq C_2$, such that if $z \in \bb{C}_v$ satisfies 
\begin{equation} \label{eq: SuperAttracting2}
\log |z-\beta|_v \leq -C_1[K:\bb{Q}](h(\vphi) + h(\beta)+1),
\end{equation}
then
$$\left| \log |\vphi(z)-\vphi(\beta)|_v - k \log |z-\beta|_v \right| \leq C_2 [K:\bb{Q}](h(\vphi) + h(\beta) + 1)$$
where $k$ is the order of vanishing of $\beta$ for the equation $\vphi(z) = \vphi(\beta)$. 
\end{proposition}

\begin{proof}
For simplicity, let $M = [K;\bb{Q}](h(\vphi) + h(\beta)+1)$. We expand 
$$\vphi(x) = \vphi(\beta) + d_1(z-\beta) + d_2(z-\beta)^2 + \cdots.$$
Let $d_k$ be the smallest $k$ for which $d_k \not = 0$. Then $k \leq 2d-2$ and 
$$h(d_k) \leq C(h(\vphi) + h(\beta)+1)$$
where $C$ is the constant from Proposition \ref{DerivativeBound1}. Thus 
\begin{equation} \label{eq: SuperAttracting6}
\log |d_k|_v \geq - CkM.
\end{equation}
Similarly for all $n \geq k$, we have 
\begin{equation} \label{eq: SuperAttracting5}
\log |d_n|_v \leq Cn M.
\end{equation}
We may now bound 
$$\left|\left(\vphi(z) - \vphi(\beta) \right)  - d_k (z-\beta)^k \right| \leq \sum_{i=k+1}^{\infty} \left|d_i (z-\beta)^i \right|.$$
Now let
$$\log |z-\beta|_v = -C_1 M$$
then
$$|d_i (z-\beta)^i| \leq e^{-i (C_1 - C) M}.$$
Then summing up from $i=k+1$ to $\infty$ gives an upper bound 
$$\sum_{i=k+1}^{\infty} \left|d_i (z-\beta)^i \right| \leq 2 e^{-i(C_1 - C)M}.$$
Now let's assume that $C_1 \geq 3d C$, so that $i(C_1 - C) \geq (k+ \frac{1}{d})C_1$ as $i \geq k+1$ and $k \leq 2d$. Then we obtain 
\begin{equation} \label{eq: SuperAttracting4}
\left|\vphi(z) - \vphi(\beta)| - d_k |z-\beta|^k \right| \leq 2e^{-(k + \frac{1}{d})C_1 M}.
\end{equation}
By \eqref{eq: SuperAttracting5} and \eqref{eq: SuperAttracting6}, we have 
$$e^{(-kC_1 - kC)M} \leq d_k|z-\beta|^k \leq e^{(-kC_1 + Ck)M}.$$
We now assume that $C_1$ is large enough so that
\begin{equation} \label{eq: SuperAttracting7}
e^{-(k+\frac{1}{d})C_1 M} \leq \frac{1}{2} e^{M(-kC_1 - kC)} \leq \frac{1}{2} d_k |z-\beta|^k.
\end{equation}
It follows from \eqref{eq: SuperAttracting4}
$$ \frac{1}{2} |z-\beta|^k \leq |\vphi(z) - \vphi(\beta)| \leq \frac{3}{2} d_k |z-\beta|^k$$
and so
$$\left|\log |\vphi(z) - \vphi(\beta)|_v ^{-1} - k \log |z-\beta|^{-1}_v \right| \leq \log |d_k|_v + \log 2$$
which by \eqref{eq: SuperAttracting5} gives us the $C_2$ we want as desired, where we take $C_1$ to be a constant larger than $3dC$ so that \eqref{eq: SuperAttracting7} holds. 
\end{proof}

Now if say $\beta$ was fixed of multiplicity $e \geq 2$, then if $C_1$ was sufficiently large, the hypothesis \eqref{eq: SuperAttracting2} continues to hold as we iterate and replace $z$ with $\vphi(z),\vphi^2(z),\ldots$ and we obtain an accurate control. However if $\beta$ was part of a superattracting cycle instead, say with distinct points $\beta,\vphi(\beta),\ldots,\vphi^{m-1}(\beta)$, then iterating the distance of $\vphi^i(z)$ to $\vphi^i(\beta)$ may decrease as $i$ goes from $0$ to $m-1$. The next proposition controls how bad it can be. 

\begin{proposition} \label{DerivativeBound3}
Let $K$ be a number field and $\vphi$ a rational map defined over $K$ of degree $d \geq 2$. Let $\beta$ be part of a superattracting cycle for $\vphi$ with period $m$. Then for any place $v \in M_K$, there exists constants $C_1 = C_1 (d) > 0, C_2 = C_2(d) > 0$, both depending only on the degree $d$, with $C_1 \geq C_2$, such that if $z \in \bb{C}_v$ satisfies 
\begin{equation} \label{eq: SuperAttracting2}
\log |z-\beta|_v \leq -(m+1)C_1[K:\bb{Q}](h(\vphi) + h(\beta)+1),
\end{equation}
then
$$\left| \log |\vphi^m(z)-\beta|_v - e \log |z-\beta|_v \right| \leq emC_2 [K:\bb{Q}](h(\vphi) + h(\beta) + 1)$$
where $e$ is the order of vanishing of $\beta$ for the equation $\vphi^m(z) = \beta$. 
\end{proposition}

\begin{proof}
Let $e_i$ be the multiplicity of $\vphi$ at $\vphi^{i-1}(\beta)$ where $\vphi^{0}(\beta) = \beta$. Then $\prod_{i=1}^{m} e_i = e \geq 2$. We will apply Proposition \ref{DerivativeBound2} with $\h_{\vphi}(\beta)$ replacing $h(\beta)$, which we may do so because of Proposition \ref{Heights2}. Let $C_1,C_2$ be the constants from Proposition \ref{DerivativeBound2}, where we will assume that $C_1 \geq C_2$. Let $M = [K:\bb{Q}](h(\vphi) +1)$, where we omit $\h_{\vphi}(\beta)$ as it is $0$. Assume that 
$$\log |z-\beta|_v \leq - (m+1)C_1 M.$$
Then we have 
$$|\log |\vphi(z) - \vphi(\beta)|_v - e_1 \log |z-\beta|_v| \leq C_2 M.$$
and in particular 
$$\log |\vphi(z) - \vphi(\beta)|_v \leq -mC_1 M.$$
This allows us to apply Proposition \ref{DerivativeBound2} again. This gives us 
$$\left| \log |\vphi^2(z) - \vphi^2(\beta)|_v - e_1 e_2 \log |z-\beta|_v \right| \leq e_2 C_2  + C_2M.$$
Repeating this $m-1$ more times, we get
$$\left|\log |\vphi^m(z) - \vphi^m(\beta)|_v - \left( \prod_{i=1}^{m} e_i \right) \log |z-\beta|_v \right| \leq \left(\sum_{i=2}^{m} \left(\prod_{j=i}^{m} e_j \right) C_2 M \right).$$
Since $\prod_{i=1}^{m} e_i = e$, we may bound this by 
$em C_2M$ as desired.

\end{proof}

We now begin the proof of Theorem \ref{TheoremReduced1} when $\beta$ is part of a superattracting cycle. Again, we keep the following assumptions: 

\begin{enumerate}
\item We may assume that $\h_{\vphi}(\alpha) \geq h(\vphi) + 1$.

\item We may assume that 
$$\left(1+ \frac{1}{100} \right) d^n \h_{\vphi}(\alpha) \geq h(\vphi^n(\alpha)) \geq \left(1 - \frac{1}{100} \right) d^n \h_{\vphi}(\alpha)$$
for all $n \geq 1$. 

\item We may assume that $h(\alpha) \geq h(\beta)$.

\item We may assume that $\infty \not \in \{\alpha,\vphi(\alpha),\ldots\}$. 

\item We may assume that $\infty \not \in \vphi^{-i}(\beta)$ for all $1 \leq i \leq N$ where $d^{N-1} \leq 56 \cdot 2^{2d-2} |S|$.
\end{enumerate}

Let $m$ be the period of $\beta$. We first show that if $m$ is large enough, then the arguments of Section \ref{sec: NonSuperAttracting} apply.

\begin{proposition} \label{SuperAttractingBound1}
We may assume that $m \leq N$ where $N$ is the largest positive integer such that 
$$d^{N-1} \leq 56 \cdot 2^{2d-2} |S|.$$
\end{proposition}

\begin{proof}
Let's assume that $m \geq N$. Then we can apply the argument in Section \ref{sec: NonSuperAttracting}. Indeed, the only property about $\beta$ being non-superattracting that we used is  
$$e_{\vphi^n}(\beta') \leq 2^{2d-2}$$
for all $\beta' \in \vphi^{-n}(\beta)$ and for all $n \geq 1$. But in our argument for non-superattracting $\beta$, we only used this property for $1 \leq n \leq N$. If our period $m$ is $> N$, then it follows that the same multiplicity bound still holds for $n \leq N$ as for $\beta' \in \vphi^{-n}(\beta)$, no critical point appears twice in the forward orbit 
$$\{\beta', \vphi(\beta'),\ldots, \vphi^n(\beta')\}.$$
Hence the argument in Section \ref{sec: NonSuperAttracting} still applies in this case and we may assume that $m \leq N$. 
\end{proof}

Now that we have an upper bound on $m$, we can apply Proposition \ref{DerivativeBound3}. Again let $N_S$ be the set of $n$'s for which $\vphi^n(\alpha)$ is $S$-integral relative to $\beta$. For $n \in N_S$, let $Y_n$ be the subset of $v \in S$ for which 
$$\log^+|\vphi^n(\alpha) - \beta|^{-1}_v \geq 2m C_1 [K:\bb{Q}] (h(\vphi) + h(\beta)+1)$$
where $C_1$ is the constant from Proposition \ref{DerivativeBound3}.

\begin{proposition} \label{SuperAttractingBound2}
Let $n' > n$ with $m \mid n' - n$, then 
$$Y_{n} \subseteq Y_{n'}.$$
\end{proposition}

\begin{proof}
It suffices to apply Proposition \ref{DerivativeBound3}. Let $C_1,C_2$ be the two constants from Proposition \ref{DerivativeBound3}. For $v \in Y_n$, since 
$$\log^+|\vphi^n(\alpha)-\beta|^{-1}_v \geq 2m C_1 [K:\bb{Q}](h(\vphi)+ h(\beta)+1),$$
it follows that 
$$\log^+|\vphi^{n+m}(\alpha) - \beta|_v^{-1} \geq (2emC_1 - emC_2)[K:\bb{Q}](h(\vphi)+h(\beta)+1).$$
Since $C_1 \geq C_2$, this is 
$$\geq emC_1[K:\bb{Q}](h(\vphi) + h(\beta)+1)$$
and so $v \in Y_{n+m}$. Thus $v \in Y_{n'}$ if $m \mid n' - n$ as desired.
\end{proof}

On the other hand, we can show that if $n' - n$ is sufficiently large, then $Y_n$ must be a proper subset of $Y_{n'}$. This will then lead us to a contradiction since we cannot have a strictly increasing chain of $Y_n$'s of length $\geq |S|+2$ as $Y_n \subseteq S$. 

\begin{proposition} \label{SuperAttractingBound3}
There exists $n_0 = n_0(d) > 0$ and $C = C(d) > 0$ such that if $n \geq n_0 \log |S|$ and $n' - n \geq C$ with $m \mid n'-n$, then
$$Y_{n} \text{ is a proper subset of } Y_{n'}.$$
\end{proposition}

\begin{proof}
By Proposition \ref{SuperAttractingBound2}, we know that $Y_n \subseteq Y_{n'}$. Hence we may assume for a contradiction that $Y_n = Y_{n'}$. We write 
$$h(\vphi^{n'}(\alpha) - \beta) = \sum_{v \in Y_{n'}} \log^+|\vphi^{n'}(\alpha) - \beta|_v ^{-1} + \sum_{v \in S \setminus Y_{n'}} \log^+|\vphi^{n'}(\alpha) - \beta|^{-1}_v.$$
For $v \in S \setminus Y_{n'}$, we have the upper bound 
$$\log^+|\vphi^{n'}(\alpha) - \beta|_v^{-1} \leq 2m C_1[K:\bb{Q}](h(\vphi) + h(\beta)+1).$$
Now let $e$ be the multiplicity of $\vphi^m$ at $\beta$. Note that $e \leq (d-1)^m$ since $\beta$ is non-exceptional. Let $\ell = \frac{n'-n}{m}$. For $v \in Y_{n'}$, applying Proposition \ref{DerivativeBound3} $\ell$ many times gives us 
$$\log^+|\vphi^{n'}(\alpha) - \beta|_v^{-1} \leq e^{\ell} \log^+|\vphi^n(\alpha) - \beta|_v^{-1} + (e^{\ell} + (e)^{\ell-1} + \cdots + e)m C_2 [K:\bb{Q}](h(\vphi) + h(\beta)+1).$$
We may bound $\sum_{i=1}^{\ell} (e)^{i}$ by $2 e^{\ell}$ and so we obtain 
\begin{equation} \label{eq: SuperAttracting3}
\sum_{v \in S} N_v \log^+|\vphi^{n'}(\alpha) - \beta|^{-1}_v 
\end{equation}
$$\leq e^{\ell} \sum_{v \in S} \log^+|\vphi^n(\alpha) - \beta|^{-1}_v   + |S| (2m e^{\ell}) C_2 [K:\bb{Q}](h(\beta) + h(\vphi) +1).$$
On the other hand, 
$$\sum_{v \in S} N_v \log^+|\vphi^{n'}(\alpha) - \beta| = h(\vphi^{n'}(\alpha) - \beta) \geq d^{n'-1} h(\alpha)$$
and 
$$e^{\ell} \sum_{v \in S} \log^+|\vphi^n(\alpha) - \beta|^{-1}_v = e^{\ell} h(\vphi^n(\alpha) - \beta) \leq e^{\ell} d^{n+1} h(\alpha).$$
We thus obtain from $\eqref{eq: SuperAttracting3}$ the inequality 
$$\left(d^{n'-1} - e^{\ell} d^{n+1} \right) h(\alpha) \leq |S|(2 m e^{\ell + 1}) C_2 [K:\bb{Q}] (h(\beta) + h(\vphi) +1).$$
Using $e \leq (d-1)^m$, we may write this as 
$$\left(d^{n'-1} - (d-1)^{n'-n} d^{n+1} \right) h(\alpha) \leq |S| (2 (d-1)^{n'-n} m) C_2 [K:\bb{Q}] (h(\beta) + h(\vphi) +1).$$
We now note that if $n'-n \geq C$ for $C = C(d) > 0$, we have 
$$(d-1)^{n'-n} d^{n+1} \leq \frac{1}{2} d^{n'-1}$$
and so we get 
$$\frac{1}{2} d^{n'-1} h(\alpha) \leq  2 |S| d^{n'-n} m C_2 [K:\bb{Q}](h(\beta) + h(\vphi) + 1)$$
where we replaced $(d-1)^{n'-n} $ by $d^{n'-n} $. Finally this is a contradiction for $n' \geq n_0 \log |S|$ where $n_0$ is some large constant depending only on $d$ as we may bound $[K:\bb{Q}]$ from above by $2|S|$ and $d \geq 2$. This finishes the proof.
\end{proof}

We can now conclude the proof of Theorem \ref{TheoremReduced1} in the case where $\beta$ is superattracting. 

\begin{proof}[Proof of Theorem \ref{TheoremReduced1} when $\beta$ is superattracting] Again let $N_S$ be the set of $n$'s for which $\vphi^n(\alpha)$ is $S$-integral relative to $\beta$. First, if $m$ is the period of $\beta$, then we have 
$$m \leq 2d + \log |S| + \log (56d).$$
Let $n_0,C$ be the constants in Proposition \ref{SuperAttractingBound3}. Consider $n$'s sufficiently large for which $n \geq n_0 \log |S|$ where $n_0$ is the constant in Proposition \ref{SuperAttractingBound3}. Then if we had more than $m C(|S|+1)$ such $n$'s, we must be able to find $|S|+2$ many of them, say $n_1,\ldots,n_{|S|+2}$ such that $m \mid n_{i+1} - n_i$ and $n_{i+1} - n_i \geq C$. Applying Proposition \ref{SuperAttractingBound3} tells us that we have the containments
$$Y_{n_1} \subset Y_{n_2} \subset \cdots \subset Y_{n_{|S|+2}}$$
where each containment is proper. However each $Y_{n_i}$ is a subset of $S$ and so can have at most $|S|$ many elements. This is a contradiction as desired and so we can have at most $mC(|S|+1)$ many elements which is of the form $c |S|(\log |S|)$ where $c = c(d) > 0$ depends only on $d$ as desired.
\end{proof}

\section{Uniformity Statements} \label{sec: Uniform}
We will prove Theorem \ref{IntroTheorem4} in this section. We first show that the dependency of the upper bound in Theorem \ref{IntroTheorem3} on $h(\vphi)/\h_{\vphi}(\alpha)$ is necessary.

\begin{proposition}
Let $S = \{\infty\}$ where $\infty$ is the archimedean place of $\bb{Q}$. Then there exists a sequence of rational maps $(\vphi_m)$ of fixed degree $d \geq 2$ defined over $\bb{Q}$, with potentially good reduction at all finite places $v$, and sequence of points $(\alpha_m), (\beta_m)$ of $\bb{P}^1(\bb{Q})$ such that

\begin{enumerate}
\item $h(\vphi_m)/h_{\vphi_m}(\alpha_m) \to \infty,$
\item There  exists a constant $C > 0$ such that
$$\{n \geq 0 \mid \vphi_m^n(\alpha_m) \text{ is S-integral relative to } \beta_m\} \geq \log^+_d \left( \frac{h(\vphi_m) + 1}{h_{\vphi_m}(\alpha_m)} \right) - C.$$
\end{enumerate}
\end{proposition}

\begin{proof}
Let $\vphi(z) = z^2+1$. Then $0$ is not preperiodic. We will consider affine transformations $\psi_m = a_m z + 1$ with $a_m \in \bb{Q}$ and set 
$$\vphi_m = \psi_m \circ \vphi \circ  \psi_m^{-1}, \quad \alpha_m = \beta_m = \psi_m(0) = 1.$$
Then $h_{\vphi_m}(\alpha_m) = \h_{\vphi}(0)$ remains constant and $\vphi_m^n(\alpha_m) = \psi_m( \vphi^n(0))$. Now consider the first $m$ points on the orbit $\vphi(0),\vphi^2(0),\ldots,\vphi^m(0)$ and choose $a_m$ such that $|a_m \vphi^i(0)|_v \geq 1$ for all $1 \leq i \leq m$ and finite places $v$. We claim that $\vphi_m^i(\alpha_m)$ is $S$-integral relative to $\beta_m$ for all $1 \leq i \leq m$. Indeed, since $\beta_m = 1$, we have that $\vphi_m^i(\alpha_m)$ is not $S$-integral relative to $\beta_m$ only if 
$$|\vphi_m^i(\alpha_m) - 1|_v < 1 \implies |\psi_m(\vphi^i(0)) - 1|_v < 1 \implies |a_m \vphi^i(0)|_v < 1$$
for some finite place $v$, which is not the case due to our choice of $a_m$. It suffices to estimate how large $h(a_m)$ is. We let $a_m$ be the reciprocal of the product of the numerators of each $\vphi^i(0)$ for $1 \leq i \leq m$. This gives us
$$h(a_m) \leq \sum_{i=1}^{m} h(\vphi^i(0)).$$
By Proposition \ref{Heights2}, there exists $c > 0$ such that
$$h(\vphi^i(0)) \leq 2^i \h_{\vphi}(0) + c \text{ for all } i \in \bb{N}.$$
Hence
$$h(a_m) \leq 2^{m+1} \h_{\vphi}(0) + mc.$$
Proposition 5(c) of \cite{HS11} implies that 
$$h(\vphi_m) \leq h(\vphi) + 4(2^{m+1} \h_{\vphi}(0) + \log 8 + mc)$$
and setting $d = 2$, we get
$$\log_d^+ \left( \frac{h(\vphi_m) + 1}{\h_{\vphi}(0)} \right) \leq \log_2^+ \left( \frac{ 4(2^{m+1} \h_{\vphi}(0) + \log 8 + mc)+1}{\h_{\vphi}(0)} \right) \leq m+C $$
for some constant $C > 0$ as desired.

\end{proof}

The fact that we cannot have a bound on the number of $S$-integral points in orbits that is uniform over all rational maps of degree $\leq d$ has already been noted in \cite{KLS15}. They instead consider the question of $S$-units in an orbit, and conjecture that there a uniform bound should exist for all rational maps of degree $\leq d$. We now prove the existence of a uniform bound for all polynomials with bounded number of places of bad reduction. 
\par 
We first state a theorem of Looper which bounds minimum value of $\h_{f}(\alpha)$ in terms of $h(f)$ uniformly for polynomials $f$ with bounded places of bad reduction. This was proven first for unicritical maps by Ingram \cite{Ing09b} and has recently been generalized to rational maps by the author \cite{Yap25b}. For $\mathbf{c} = (c_1,\ldots,c_{d-1})$, set
$$f_{\bf{c}} = \frac{1}{d} z^d - \frac{1}{d-1} (c_1 + \cdots + c_{d-1})z^{d-1} + \cdots + (-1)^{d-1}c_1 c_2 \cdots c_{d-1}z.$$

\begin{theorem} \cite[Theorem 1.2]{Loo19} \label{Looper1}
Suppose $K$ is a number field and $\mathbf{c} \in K^{d-1}$ with $f_{\bf{c}}$ having at most $s$ places of bad reduction. Then there exists $\kappa_1, \kappa_2 > 0$, depending only on $d,s,[K:\bb{Q}]$ such that
$$h_{f_{\bf{c}}}(\alpha) \geq \kappa_1 h(f_{\bf{c}}) + \kappa_2$$
for all $\alpha \in K$ that is not preperiodic. 
\end{theorem}

Since every polynomial $f \in K[z]$ is conjugate to $f_{\bf{c}}$ for some $\bf{c}$, with the same number of places of bad reduction, we can conjugate $f$ to $f_{\bf{c}}$ and apply Theorem \ref{IntroTheorem3} to it. However, conjugating might change the $S$-integrality of our points. To overcome this, we use Hsia--Silverman's notion of quasi-$(S,\epsilon)$-integrality. A point which is not quasi-$(S,\epsilon)$-integral can be thought of being at least $(1-\epsilon)$-far away from being $S$-integral and in particular, gives us some lee way when conjugating. 
\par 
We now recall the definition of a quasi-$(S,\epsilon)$-integral point. Fix an $\epsilon$ with $0 \leq \epsilon \leq 1$. Let $S$ be a finite set of places of $K$ containing all archimedean places. We say that $x \in K$ is quasi-$(S,\epsilon)$-integral if
$$\sum_{v \in S} N_v \log^+|x|_v \geq \epsilon h(x).$$
When $\epsilon = 1$, we have
$$\sum_{v \in S} N_v \log^+|x|_v = h(x)$$
which implies that $\log^+|x|_v = 0$ for all $v \not \in S$. In particular we recover the definition of $S$-integral points. If $\frac{1}{\vphi^n(\alpha)-\beta}$ is quasi-$(S,\epsilon)$-integral, we say that $\vphi^n(\alpha)$ is quasi-$(S,\epsilon)$-integral relative to $\beta$.  
\par 
For a rational map $\vphi: \bb{P}^1 \to \bb{P}^1$ of degree $d \geq 2$ defined over $K$ and $\alpha,\beta \in \bb{P}^1(K)$, we set 
$$\Gamma_{\vphi,S}(\alpha,\beta,\epsilon) = \{ n \geq 0 \mid (\vphi^n(\alpha) - \beta)^{-1} \text{ is quasi-}(S,\epsilon)\text{-integral}.\}$$
\par 
The following theorem is proven in \cite{HS11}, with one modification to remove the dependency on $\h_{\vphi}(\beta)$. 

\begin{theorem} \label{QuasiSIntegral1}
Fix an $\epsilon$ with $0 \leq \epsilon < 1$. Assume that $\alpha$ is not preperiodic and $\beta$ is non-exceptional. Then there exists $\gamma = \gamma([K:\bb{Q}],d,\epsilon)$ such that
$$|\Gamma_{\vphi,S}(\alpha,\beta,\epsilon)| \leq (\gamma) \cdot 4^{|S|} + \log^+_d \left(\frac{h(\vphi)+1}{\h_{\vphi}(\alpha)}\right).$$
\end{theorem}

\begin{proof}
Theorem 11(b) of \cite{HS11} tells us that
$$|\Gamma_{\vphi,S}(\alpha,\beta,\epsilon)| \leq (\gamma) \cdot 4^{|S|} + \log^+_d \left(\frac{h(\vphi) + \h_{\vphi}(\beta)}{\h_{\vphi}(\alpha)}\right).$$
We now have to explain why we can remove $\h_{\vphi}(\beta)$ from the numerator and replace it with $1$. This follows from an analogous statement of Proposition \ref{Assumption5}, but for quasi-$(S,\epsilon)$-integral points. As the proof is very similar, we will not repeat it here. This allows us to assume that $\h_{\vphi}(\alpha) \geq \h_{\vphi}(\beta)$ and 
thus get rid of the term $\h_{\vphi}(\beta)$ in the numerator, giving us
$$|\Gamma_{\vphi,S}(\alpha,\beta,\epsilon)| \leq (\gamma) \cdot 4^{|S|} + \log^+_d \left(\frac{h(\vphi)+1}{\h_{\vphi}(\alpha)}\right)$$
as desired.
\end{proof}

Before we start our proof of Theorem \ref{IntroTheorem4}, we first analyze how replacing $(\alpha,0)$ by $(\psi(\alpha),\psi(0))$, where $\psi$ is an affine linear transformation, can affect non-quasi-$(S,\epsilon)$-integrality of $\alpha$. For $x \in K$, we will let 
$$h_S(x) = \sum_{v \not \in S} N_v \log^+|x|_v$$
for convenience. 
\begin{proposition} \label{Conjugacy1}
Assume that $\epsilon \leq \frac{1}{10}$. Let $\psi = az - b$ with $a,b \in K$ and $a \not = 0$. Let $\alpha_1,\alpha_2 \in K$ be two elements that are quasi-$(S,\epsilon)$-integral relative to $\frac{b}{a}$ and assume $h(\alpha_1) \geq h(\alpha_2) + 3 \log 2$. Then if $\psi(\alpha_1), \psi(\alpha_2)$ are both $S$-units, we must have
$$ 9 h(\alpha_1) \geq h\left(\frac{b}{a}\right) + h_S(a) \geq \frac{1}{2} h(\alpha_1).$$
\end{proposition}

\begin{proof}
We have $\psi(\alpha_i) = a \alpha_i - b$. Then being a $S$-unit implies that $\log |a \alpha_i-b|_v = 0$ for all $v \not \in S$. Hence for each $v \not \in S$ where $|\alpha_i - \frac{b}{a}|_v < 1$, we must have $|a|_v = |\alpha_i - \frac{b}{a}|^{-1}_v$.
Hence we have 
$$h_S(a) = \sum_{v \not \in S} N_v\log^+|a|_v = \sum_{v \not \in S} N_v \log^+ \left| \alpha_i - \frac{b}{a} \right|^{-1}_v.$$
As $\alpha_i$ is quasi-$(S,\epsilon)$-integral relative to $\frac{b}{a}$, we get 
\begin{equation} \label{eq: ConjugacyEqn1} 
h_S(a) = \sum_{v \not \in S} N_v \log^+ \left| \alpha_i - \frac{b}{a} \right|^{-1}_v \geq (1-\epsilon)h\left(\alpha_i - \frac{b}{a} \right) \geq \frac{9}{10} \left(h(\alpha_i) - h\left(\frac{b}{a} \right) - \log 2 \right).
\end{equation}
As $h(\alpha_1) \geq 3 \log 2$, we obtain  
$$h_S(a) + h\left(\frac{b}{a} \right) \geq \frac{1}{2} h(\alpha_1)$$
which gives our lower bound.
\par 
For the upper bound, as $v$ is non-archimedean, from $\left|\alpha_i - \frac{b}{a} \right|^{-1}_v = |a|_v$, we obtain 
$$\left|\alpha_i - \frac{b}{a} \right|_v = |a|^{-1}_v \implies |\alpha_1 - \alpha_2| \leq |a|^{-1}_v$$
and so
$$|a|_v \leq |\alpha_1 - \alpha_2|_v^{-1}.$$
Hence 
$$\sum_{v \not \in S} N_v \log^+|\alpha_1 - \alpha_2|^{-1}_v \geq \sum_{v \not \in S} N_v \log^+|a|_v = h_S(a).$$
Using \eqref{eq: ConjugacyEqn1} and using $h(x-y) \geq h(y) - h(x) - \log 2$ instead, we get
$$\sum_{v \not \in S} \log^+|\alpha_1 - \alpha_2|_v^{-1} \geq  h_S(a) + \frac{9}{10} \left( h\left(\frac{b}{a} \right) - h(\alpha_i) - \log 2 \right).$$
Since 
$$\sum_{v \not \in S} \log^+|\alpha_1 - \alpha_2|_v^{-1} \leq h(\alpha_1 - \alpha_2) \leq h(\alpha_1) + h(\alpha_2 ) + \log 2 \leq 2 h(\alpha_1),$$
we get 
$$3h(\alpha_1) \geq  h_S(a) + \frac{1}{2} h \left(\frac{b}{a} \right)$$
which gives us our upper bound.

\end{proof}

Proposition \ref{Conjugacy1} essentially tells us that if applying $\psi$ changes a non-quasi-$(S,\epsilon)$-integral point $\alpha$ to a $S$-unit, then there are constrains on $h(\alpha)$. 
\par 
We now prove our uniform bound of the number of $S$-units in our orbit for polynomials with at most $s$ places of bad reduction.

\begin{theorem} \label{UniformSIntegral1}
Let $\vphi(z) \in K[z]$ be a polynomial of degree $d$ with at most $s$ places of bad reduction. Assume that $\infty$ is the only exceptional point for $\vphi(z)$. Then there exists a constant $c_6 = c_6([K:\bb{Q}],d,s,|S|)$ such that
$$\{n \geq 0 \mid \vphi^n(\alpha) \text{ is a S-unit } \} \leq c_6.$$
\end{theorem}

\begin{proof}
Let $\psi$ be an affine linear map so that $\psi^{-1} \circ \vphi \circ \psi$ is of the form 
$$\vphi_c = \frac{1}{d} z^d + a_{d-1} z^d + \cdots + a_0.$$
Then $\vphi_c$ has at most $s$ places of bad reduction and we may apply Theorem \ref{Looper1} to obtain $\kappa_1,\kappa_2 > 0$, depending on $d,s,[K:\bb{Q}]$, so that
\begin{equation} \label{eq: LowerBound1}
\h_{\vphi_c}(z) \geq \kappa_1 h(f_c) + \kappa_2
\end{equation}
for all $z \in \bb{P}^1(K)$ which is not preperiodic for $\vphi_c$. Now let $\alpha' = \psi^{-1}(\alpha)$. We note that 
$$\vphi^n(\alpha)) = \psi(\vphi_c^n(\alpha'))$$
which allows us to apply Proposition \ref{Conjugacy1}. Using \eqref{eq: LowerBound1}, for some $n_0 = n_0([K:\bb{Q}],s,d)$, if $n \geq n_0$ we have 
$$\h_{\vphi_c}(\vphi_c^n(\alpha')) \geq C(h(\vphi_c)+1)$$
where $C = C(d) > 0$ is the constant from Proposition \ref{Heights2}. Increasing $n_0$, and replacing $\alpha$ with $\vphi^{n_0}(\alpha)$, we may assume that 
\begin{equation} \label{eq: Condition1}
h(\alpha') \geq 6 \log 2
\end{equation}
and 
\begin{equation} \label{eq: Condition2}
h(\vphi^{n+1}(\alpha')) \geq \frac{3}{2} h(\vphi^n(\alpha'))
\end{equation}
for all $n \geq 0$. We are now ready to prove our theorem. Let $N_S$ be the set of $n$'s for which $\vphi^n(\alpha)$ is a $S$-unit. Let $\psi = az-b$ and let $\beta = \frac{b}{a}$. Since $\beta$ is not exceptional, if we take $\epsilon = \frac{1}{10}$ then by Theorem \ref{QuasiSIntegral1}, there are only finitely many $n$'s, depending only on $M = M([K:\bb{Q}],d,|S|)$ for which $\vphi_c^n(\alpha')$ is quasi-$(S,\epsilon)$-integral relative to $\beta$. We denote this set by $M_S$.
\par 
We now consider $n \in N_S \setminus M_S$. Since $\vphi^n(\alpha) = \psi(\vphi_c^n(\alpha'))$, by Proposition \ref{Conjugacy1} and \eqref{eq: Condition1}, we have 
$$9h(\vphi_c^n(\alpha')) \geq h\left(\frac{b}{a} \right) + h_S(a) \geq \frac{1}{2} h(\vphi_c^n(\alpha')).$$
By \eqref{eq: Condition2}, if we have $n',n \in N_S \setminus M_S$, it must be that 
$$n' - n \leq \log_{3/2}(18).$$
Hence $|N_S \setminus M_S|$ is a finite set which finishes our proof. 
\end{proof}

\bibliography{bibfile}
\bibliographystyle{alpha}

\end{document}